\documentclass[12pt]{amsart}
\usepackage{amssymb}
\usepackage{amsmath}
\usepackage{setspace}
\usepackage{mathtools}
\usepackage{verbatim}
\usepackage{lipsum}
\usepackage{amsthm}
\usepackage{tikz-cd}
\usepackage{hyperref}

\newtheorem{theorem}{Theorem}[section]
\newtheorem{lemma}[theorem]{Lemma}
\newtheorem{corollary}[theorem]{Corollary}

\newtheorem{remark}[theorem]{Remark}

\newtheorem{proposition}[theorem]{Proposition}
\newtheorem{defn}[theorem]{Definition}

\newtheorem{example}[theorem]{Example}

\newcommand{\bb}[1]{\mathbb{#1}}

\newcommand{\cal}[1]{\mathcal{#1}}

\makeatletter
\newcommand*{\da@rightarrow}{\mathchar"0\hexnumber@\symAMSa 4B }
\newcommand*{\da@leftarrow}{\mathchar"0\hexnumber@\symAMSa 4C }
\newcommand*{\xdashrightarrow}[2][]{%
  \mathrel{%
    \mathpalette{\da@xarrow{#1}{#2}{}\da@rightarrow{\,}{}}{}%
  }%
}
\newcommand{\xdashleftarrow}[2][]{%
  \mathrel{%
    \mathpalette{\da@xarrow{#1}{#2}\da@leftarrow{}{}{\,}}{}%
  }%
}
\newcommand*{\da@xarrow}[7]{%
  \sbox0{$\ifx#7\scriptstyle\scriptscriptstyle\else\scriptstyle\fi#5#1#6\m@th$}%
  \sbox2{$\ifx#7\scriptstyle\scriptscriptstyle\else\scriptstyle\fi#5#2#6\m@th$}%
  \sbox4{$#7\dabar@\m@th$}%
  \dimen@=\wd0 %
  \ifdim\wd2 >\dimen@
    \dimen@=\wd2 %
  \fi
  \count@=2 %
  \def\da@bars{\dabar@\dabar@}%
  \@whiledim\count@\wd4<\dimen@\do{%
    \advance\count@\@ne
    \expandafter\def\expandafter\da@bars\expandafter{%
      \da@bars
      \dabar@ 
    }%
  }%
  \mathrel{#3}%
  \mathrel{%
    \mathop{\da@bars}\limits
    \ifx\\#1\\%
    \else
      _{\copy0}%
    \fi
    \ifx\\#2\\%
    \else
      ^{\copy2}%
    \fi
  }%
  \mathrel{#4}%
}
\makeatother

\usepackage[french, english]{babel}

\newcommand{\newabstract}[1]{%
  \par\bigskip
  \csname otherlanguage*\endcsname{#1}%
  \csname captions#1\endcsname
  \item[\hskip\labelsep\scshape\abstractname.]
}

\title{Higher dimensional foliated Mori theory}
\author{Calum Spicer}
\email{calum.spicer@imperial.ac.uk}
\address{Department of Mathematics, Imperial College London}


\begin{document}

\begin{abstract}
We develop some foundational results in a higher dimensional foliated Mori theory, and
show how these results can be used to prove a structure theorem for the Kleiman-Mori cone
of curves in terms of the numerical properties of $K_{\mathcal{F}}$ for rank 2 foliations
on threefolds.  We also make progress
toward realizing a minimal model program (MMP) for rank 2 foliations on threefolds.


\end{abstract}
\maketitle

\tableofcontents

\section{Introduction}

We will always work over $\bb C$.
Let $X$ be a normal variety and let $\cal F$ be a foliation on $X$.
We define $-c_1(\cal F) = K_{\cal F}$ to be the
canonical divisor of the foliation.  
In recent years much work
has been done understanding the birational geometry of the foliation
in terms of $K_{\cal F}$ when the rank of $\cal F$ is 1, especially in the case
of rank 1 foliations on surfaces, see for example 
\cite{McQ05}, \cite{McQuillan08}, \cite{Brunella99}.  
By contrast, the case of of higher rank foliations remains
relatively unexplored. 
The goal of this paper is to begin developing the theory in the
case of co-rank 1 foliations,
especially in the case of threefolds.

An essential first step in understanding the birational geometry
of a variety or foliation is a structure theorem
on the closed cone of curves $\overline{NE}(X)$.
Our first main result is the following foliated cone theorem:

\begin{theorem}
Let $X$ be a projective
$\bb Q$-factorial and klt threefold
and $\cal F$ a co-rank 1 foliation
with non-dicritical foliation singularities.  Suppose that $(\cal F, \Delta)$
has canonical singularities.
Then
$$\overline{NE}(X) = \overline{NE}(X)_{(K_{\cal F}+\Delta)\geq 0}+\sum \bb{R}_
+ [L_i]$$
where $L_i$ is a rational curve tangent to $\cal F$ 
with 
$(K_{\cal F}+\Delta)\cdot L_i \geq -6.$

In particular, the $K_{\cal F}+\Delta$-negative
extremal rays are locally discrete
in the $(K_{\cal F}+\Delta)<0$ portion of the cone.
\end{theorem}

This result is perhaps surprising and interesting in its own right.
It is difficult to determine which line bundles on $X$ 
can be realized as the canonical bundle of a foliation: this result
places necessary conditions on such a line bundle which have a strong geometric flavor.

With the cone theorem in hand we then turn to the problem of 
constructing minimal models.  That is, given a variety and a foliation $(X, \cal F)$
is there a sequence of birational modifications that can be performed 
resulting in a model $(Y, \cal G)$ with $K_{\cal G}$ nef?

Toward this we prove a contraction theorem.
Given an extremal ray $R \subset \overline{NE}(X)$ 
we say that $c:X \rightarrow Y$
is a contraction of $R$ provided $c(\Sigma)$ is a point if and only
if $[\Sigma] \in R$:

\begin{theorem}
Suppose $X$ is a projective $\bb Q$-factorial and klt threefold.  Suppose that $(\cal F, \Delta)$
has canonical, non-dicritical singularities and that $(\cal F, \Delta)$ is terminal along
$\text{sing}(X)$.

Let $R$ be a $K_{\cal F}+\Delta$-negative extremal ray.  Then there is a contraction
of $R$ which only contracts curves tangent to $\cal F$.
\end{theorem}

Unfortunately, we are not yet able to prove the existence of minimal models in full generality.
The difficulty, as one might expect, is to prove
the existence and termination of flips.  We are able to show existence and termination
in the smooth situation, which seems to already be useful for classification problems:

\begin{theorem}
An MMP for $\cal F$ exists if
$\cal F$ is a smooth rank 2 foliation on a smooth 3-fold $X$.
\end{theorem}

It is also possible 
to prove a toric foliated MMP:

\begin{theorem}
Let $\cal F$ be a co-rank 1 toric foliation on a $n$-dimensional toric variety $X$.
Suppose that $\cal F$ has canonical, non-dicritical singularities.
Then there is an MMP for $\cal F$.
\end{theorem}

We are also able to prove the following existence of flips result:

\begin{theorem}
Let $X$ be a $\bb Q$-factorial and klt threefold and $\cal F$ a co-rank 1 foliation
with terminal singularities.
Let $f:X \rightarrow Z$ be a flipping contraction, then the flip of $f$ exists.
\end{theorem}

As a quick application of the smooth MMP we get:

\begin{corollary}
Let $X$ be a smooth threefold and let $\cal F$ be a smooth rank 2 foliation
on $X$ with $K_{\cal F}$ pseudo-effective and 
$\nu(K_{\cal F}) = 0$.  Then $(X, \cal F)$ is birational to $(Y, \cal G)$
where $(Y, \cal G)$ is one of the following:

\begin{enumerate}
\item Up to a finite cover $\cal G$ is the product of a fibration in Calabi-Yau manifolds
and a linear foliation on a torus,

\item $\cal G$ is transverse to a $\bb P^1$-bundle,

\item Up to a finite cover $\cal G$ is induced by a fibration $X' \times C \rightarrow X$
where $c_1(X') = 0$.

\end{enumerate}

\end{corollary}

We briefly explain the proofs of the cone and contraction theorems:

For simplicity, let us assume that $X$ is smooth.
Let $R$ be an extremal ray of the cone of curve $\overline{NE}(X)$
with $R \cdot K_{\cal F}<0$.
Suppose that $H_R$ is a supporting hyperplane of $R$.

$H_R$ is a nef divisor on $X$.  If $H_R^k = 0$ for some $k\leq n$
then we can show by a foliated bend and break result that
through a general point of $X$ there is a rational curve tangent to the foliation
spanning $R$.

If $H_R^n \neq 0$, then we may take $H_R$ to be effective, and so we see
that $R$ actually comes from a lower dimensional subvariety $S$ of $X$.  The idea
here is to proceed by induction on dimension.
Unfortunately, the singularities
which arise in this induction step are worse than the singularities of $(X, \cal F)$.  
Indeed, as we will see
the singularities in our induction are sometimes worse than log canonical.
The bulk of our work is therefore to surmount these difficulties.

The strategy around this will be to 
\begin{enumerate}
\item first show that the extremal ray 
is spanned by a curve then,

\item 
using a foliated subadjunction result 
(extending Kawamata's subadjunction result to the foliated case) 
we show
that the extremal ray is spanned by a curve tangent to the foliation.  Finally,

\item we develop an algebraicity criterion for 2 dimensional leaves of
foliations, which allows us to show that the extremal ray is spanned by a 
rational curve tangent to the foliation.
\end{enumerate}

The subadjunction and algebraicity results alluded to are Theorem \ref{foliatedsubadjunction} 
and
Lemma \ref{algebraicorrational} respectively.
An amusing corollary of our algebraicity criterion is the following:
\begin{corollary}
Let $\cal F$ be a smooth rank 2 foliation on a smooth projective variety.
Let $C$ be a $K_{\cal F}$-negative curve tangent to $\cal F$.
Let $S$ be a germ of a leaf around $C$.  Then either
\begin{enumerate}
\item $C$ is rational or
\item $S$ is algebraic (i.e. its Zariski closure is a surface).
\end{enumerate}
\end{corollary}

For the contraction and flip theorem, the rough idea is to deduce the foliated statement
by running a well chosen log MMP.  In proving the flip theorem we prove a version
of Malgrange's theorem for singular varieties, which we think is 
of independent interest, see Lemma \ref{novelmalgrange}:

\begin{theorem}
Let $0 \in X$ be a threefold germ with a co-rank 1 foliation $\cal F$.
Suppose $X$ has log terminal singularities
and that $\cal F$ has terminal singularities.  Then $\cal F$ has a holomorphic
first integral.
\end{theorem}

We make one final remark: modern methods
for proving the statements in the classical MMP are largely cohomological, i.e. utilising 
various theorems
on the vanishing and non-vanishing of cohomology groups. These are all 
false in the foliated context.
The main challenge in this paper 
is to develop techniques for foliations which can replace the cohomological
methods developed in the classical setting.

\subsection{Acknowledgements} 
We'd like to thank Jorge Vit\'{o}rio Pereira, 
Michael McQuillan and James McKernan for
many helpful discussions
and suggestions, as well as for providing feedback and corrections 
on earlier drafts
of this work.  Their generosity in sharing their guidance and expertise has been
greatly appreciated.
We'd also like to thank the referee for his or her careful reading and numerous suggestions which greatly
improved this paper.

\section{Set up and basic results}

\begin{defn}
Given a normal variety $X$ a {\bf foliation} $\cal F$ is a coherent saturated
subsheaf
of the tangent sheaf of $X$ which is closed under the Lie bracket.

The {\bf rank} of the foliation, $\text{rk}(\cal F)$,
is its rank as a sheaf and its {\bf co-rank} is $\text{dim}(X) - \text{rk}(\cal F)$.

The {\bf singular locus} of the foliation is the locus where $\cal F$
fails to be a sub-bundle of $T_X$.  Note that $\text{sing}(\cal F)$
has codimension at least 2.
\end{defn}

The canonical divisor plays a central role in the birational geometry
of foliations, we define it as follows:

\begin{defn}
Let $U$ be the locus where $X$ and $\cal F$ are smooth.
We can associate a divisor to
$\text{det}(\cal F\vert_U)^*$, which gives a Weil divisor on all of $X$,
denoted $K_{\cal F}$.
\end{defn}

For the rest of this paper we will take $\cal F$ to be a co-rank 1
foliation over $\bb C$.

\begin{defn}
We say $W \subset X$ is {\bf tangent} to $\cal F$ if the
tangent space of $W$ factors through $\cal F$
along $X - (\text{sing}(\cal F) \cup \text{sing}(X))$. Otherwise we say that $W$ 
is {\bf transverse} to the foliation.


If $\cal F$ factors through the tangent space of $W$, 
$\cal F\vert_W \rightarrow T_W \rightarrow T_X\vert_W$, we say
that $W$ is {\bf invariant}.
\end{defn}

\subsection{1-forms and pulled back foliations}
\begin{defn}
Let $\omega$ be a rational 1-form
with $\omega\wedge d\omega = 0$.  Then we can define a foliation
by contraction.  Namely, we take $\cal F$
to be the kernel of the pairing of $\omega$ with $T_X$.
Thus, given a rank 1 coherent subsheaf of $\Omega_X$, we can 
define
a foliation by contraction.

On the other hand, given a foliation $\cal F$ we can define a subsheaf
of $\Omega_X$ by taking the kernel of $\Omega_X \rightarrow \cal F^*$.
\end{defn}

Let $\cal F$ be a co-rank 1 foliation on $X$,
and suppose that it is defined by the
rank 1 subsheaf of the cotangent sheaf $0 \rightarrow \cal L \rightarrow \Omega_X$.

\begin{defn}
Let $f:W \rightarrow X$.  We have a morphism
$df: f^*\Omega_X \rightarrow \Omega_W$.
Assume that $f(W)$ is is not tangent to $\cal F$ and that
$f(W)$ is not contained in $\text{sing}(\cal F)\cup\text{sing}(X)$.
Then $df(f^*\cal L)$ is a rank 1 coherent subsheaf
of $\Omega_W$.  
Observe that if $\omega \in \cal L$ is
an integrable 1 form then
$df(\omega)$ is still integrable.

This gives a foliation $\cal F_W$,
called the {\bf pulled back} foliation.  
\end{defn}
When $f$ is a closed immersion
we will sometimes refer to it as the restricted foliation.

In general, even if $\cal L$ is a saturated subsheaf, $f^*\cal L$ might not be saturated.

\begin{defn}
Let $0 \rightarrow \cal L \rightarrow \Omega_X$ define
a foliation, $\cal F$.  We call the saturation of $\cal L$
in $\Omega_X$ the {\bf conormal sheaf}, denoted $N^*_{\cal F}$.

On the smooth locus of $X$, $(N^*_{\cal F})\vert_{X^{sm}}$
is a line bundle
represented by 1-forms
with zero loci of codimension at least 2.
We can therefore associate to $N^*_{\cal F}$ a well defined
Weil divisor.
We will denote this divisor by $[N^*_{\cal F}]$.
\end{defn}

\begin{lemma}
\label{basicadjunction}
We have the following equivalence of Weil divisors:
$K_X = K_{\cal F} + [N^*_{\cal F}]$.
\end{lemma}

\subsection{Foliated Pairs and Foliation singularities}

Frequently in birational geometry it is useful to consider
pairs $(X, \Delta)$ where $X$ is a normal variety, and $\Delta$ is a $\bb Q$-Weil divisor
such that $K_X+\Delta$ is $\bb Q$-Cartier.  By analogy we define

\begin{defn}
A {\bf foliated pair} $(\cal F, \Delta)$ is a pair of a foliation and a $\bb Q$-Weil 
($\bb R$-Weil) divisor
such that $K_{\cal F}+\Delta$ is $\bb Q$-Cartier ($\bb R$-Cartier).
\end{defn}

Note also that we are typically interested only in the cases when $\Delta \geq 0$,
although it simplifies some computations to allow $\Delta$ to have negative coefficients.

Given a birational morphism $\pi: \widetilde{X} \rightarrow X$ 
and a foliated pair $(\cal F, \Delta)$ on $X$ 
let $\widetilde{\cal F}$ be the pulled back foliation on $\tilde{X}$ and $\pi_*^{-1}\Delta$ be the strict
transform.
We can write
$$K_{\widetilde{\cal F}}+\pi_*^{-1}\Delta=
\pi^*(K_{\cal F}+\Delta)+ \sum a(E_i, \cal F, \Delta)E_i,$$

\begin{defn}
We say that $(\cal F, \Delta)$ is {\bf terminal, canonical, log terminal, log canonical} if
$a(E_i, \cal F, \Delta) >0$, $\geq 0$,
$> -\epsilon(E_i)$, $\geq -\epsilon(E_i)$, respectively, where
$\epsilon(D) = 0$ if $D$ is invariant and 1 otherwise and where $\pi$
varies across all birational morphisms.

If $(\cal F, \Delta)$ is log terminal and $\lfloor \Delta \rfloor = 0$
we say that $(\cal F, \Delta)$ is {\bf foliated klt}.
\end{defn}

Notice that these notions are well defined, i.e. $\epsilon(E)$ and $a(E, \cal F, \Delta)$
are independent of $\pi$.

Observe that in the case where $\cal F = T_X$ no exceptional divisor
is invariant, i.e. $\epsilon(E)=1$, and so this definition recovers the usual
definitions of (log) terminal, (log) canonical.

We have the following nice characterization due to \cite[Corollary I.2.2.]{McQuillan08}:

\begin{proposition}
Let $0 \in X$ be a surface germ with a terminal foliation $\cal F$.  
Then there exists a smooth foliation, $\cal G$,
on a smooth surface, $Y$, and a cyclic quotient
$Y \rightarrow X$ such that $\cal F$ is the quotient of
$\cal G$ by this action.
\end{proposition}

We also make note of the following easy fact:

\begin{lemma}
Let $\pi:(Y, \cal G) \rightarrow (X, \cal F)$ be a birational morphism.
Write $\pi^*(K_{\cal F}+\Delta) = K_{\cal G}+\Gamma$. Then
$a(E, \cal F, \Delta) = a(E, \cal G, \Gamma)$ for all $E$.
\end{lemma}

\begin{remark}
Observe that if any component of $\text{supp}(\Delta)$ is foliation invariant,
then $(\cal F, \Delta)$ is not log canonical.
\end{remark}

We will also make use of the class of simple foliation singularities:

\begin{defn}
We say that $p \in X$ with $X$ smooth is a {\bf simple singularity} for $\cal F$
provided in formal coordinates around $p$ we can write the defining 1-form
for $\cal F$ in one of the following two forms, where $1 \leq r \leq n$:

(i) There are $\lambda_i \in \bb C^*$ such that
$$\omega = (x_1...x_r)(\sum_{i = 1}^r \lambda_i \frac{dx_i}{x_i})$$
and if $\sum a_i\lambda_i = 0$ for some non-negative integers $a_i$
then $a_i = 0$ for all $i$.

(ii) There is an integer $k \leq r$ such that
$$\omega = (x_1...x_r)(\sum_{i = 1}^kp_i\frac{dx_i}{x_i} + 
\psi(x_1^{p_1}...x_k^{p_k})\sum_{i = 2}^r \lambda_i\frac{dx_i}{x_i})$$
where $p_i$ are positive integers, without a common factor, $\psi(s)$
is a series which is not a unit, and $\lambda_i \in \bb C$
and if $\sum a_i\lambda_i = 0$ for some non-negative integers $a_i$
then $a_i = 0$ for all $i$.

We say the integer $r$ is the dimension-type of the singularity.
\end{defn}

\begin{remark}
A general hyperplane section of a simple singularity is again
a simple singularity.  
\end{remark}

By Cano, \cite{Cano}, every foliation on a smooth threefold admits a resolution
by blow ups centred in the singular locus of the foliation
such that the transformed
foliation has only simple singularities.


Using \cite{AD13} it is easy to check the following:

\begin{lemma}
Simple singularities (including smoothly foliated points) are canonical.
\end{lemma}

The converse of this statement is false:

\begin{example}
Consider the germ of the foliation $(0 \in X, \cal F)$ given by the
degeneration of smooth surfaces to the cone over an elliptic curve.
Consider the blow up $\pi$ at the point $0$ with
exceptional divisor $E$ and let $\cal F'$
be the transformed foliation.
Observe that $\cal F'$ has simple singularities,
and that $E$ is invariant.

Write $K_{\cal F'} = \pi^*K_{\cal F}+aE$.

Denote by $L$ the closure of a leaf in $X$ passing through $0$,
and $L'$ its strict transform.
$K_{\cal F}\vert_L = K_L$,
$K_{\cal F'}\vert_{L'} = K_{L'}+E\vert_{L'}$ and
$K_{L'} = \pi^*K_L - E\vert_{L'}$.

From this we see that
$K_{L'}+E\vert_{L'} = \pi^*K_L+aE\vert_{L'}$
and so $a = 0$, hence $\cal F$ is canonical.
However, $\cal F$ is not simple since simple
singularities in dimension $\geq 3$ are never isolated.
\end{example}

We will need to define one final type of foliation singularity:

\begin{defn}
Given a foliated pair $(X, \cal F)$ we say that 
$\cal F$ has {\bf non-dicritical} singularities if for any sequence of
blow ups $\pi:(X', \cal F') \rightarrow (X, \cal F)$ such that $X'$ is smooth
and any $q \in X$ we have $\pi^{-1}(q)$ is tangent to
the foliation.
\end{defn}

\begin{remark}
Observe that non-dicriticality implies that if $W$ is $\cal F$ invariant, then
$\pi^{-1}(W)$ is $\cal F'$ invariant.
\end{remark}

\begin{defn}
Given a germ $0 \in X$ with a foliation $\cal F$
such that $0$ is a singular point for $\cal F$
we call a (formal) hypersurface germ $0 \in S$ a {\bf (formal) separatrix}
if it is invariant under $\cal F$.
\end{defn}

Note that away from the singular locus of $\cal F$
a separatrix is in fact a leaf.  Furthermore being non-dicritical 
implies that there are only finitely
many separatrices through a singular point.

\begin{example}
Consider the foliation on $\bb C^3$ defined by the 1-form \[\omega = (x^my-z^{m+1})dx+(y^mz-x^{m+1})dy+(z^mx-y^{m+1})dz\]
where $m \geq 2$.
Blowing up the origin shows that this is a dicritical foliation.
Observe that every line passing through the origin is tangent to this foliation however it admits no separatrices at the origin.
\end{example}

\begin{example}
Let $\lambda \in \bb R$.  Consider the foliation $\cal F_\lambda$ on $\bb C^2$
generated by $x\partial_x+\lambda y\partial_y$.
For $\lambda \in \bb Q_{> 0}$ we can see that $\cal F_\lambda$ is dicritical,
and otherwise is non-dicritical.
\end{example}

\begin{example}
Simple singularities are non-dicritical.
\end{example}

Even for simple foliation singularities it is possible that there
are separatrices which do not converge.  However, as the following
definition/result of \cite{CC92} shows there is always at least 1 convergent
separatrix along a simple foliation singularity of codimension 2.

\begin{defn}
For a simple singularity of type (i), all separatrices are convergent.

For a simple singularity of type (ii), around a general point
of a codimension 2 component of the singular locus we can write 
$\omega = pydx+qxdy+ x\psi(x^py^q)\lambda dy$. $x = 0$ is a convergent
separatrix, called the {\bf strong separatrix}.
\end{defn}

\begin{defn}
\label{tangtransdef}
Suppose $X$ is a normal variety and $\cal F$ is a co-rank 1 foliation with non-dicritical singularities.

We say $W \subset X$ (possibly contained in $\text{sing}(X)$)
is {\bf tangent} to the foliation if for any birational morphism 
$\pi:(X', \cal F') \rightarrow (X, \cal F)$ and any (equivalently some) 
$\pi$-exceptional divisor $E$ such that $E$
dominates $W$ we have that $E$ is $\cal F'$-invariant.

We say $W \subset X$ (possibly contained in $\text{sing}(X)$)
is {\bf transverse} to the foliation if for any birational morphism 
$\pi:(X', \cal F') \rightarrow (X, \cal F)$ and any (equivalently some) 
$\pi$-exceptional divisor $E$ such that $E$
dominates $W$ we have that $E$ is not $\cal F'$-invariant.
\end{defn}

Notice that this definition agrees with the one given earlier if
$W$ is not contained in $\text{sing}(X) \cup \text{sing}(\cal F)$.

\subsection{Foliated MMP for surfaces}

As mentioned earlier, McQuillan in \cite{McQuillan08} proves the existence of a foliated MMP,
namely:
\begin{theorem}
Let $X$ be a smooth projective
surface and $\cal F$ a foliation with canonical foliation singularities.
Then, there is an MMP starting with $X$, namely a sequence of contractions of curves
$\pi:X \rightarrow Y$ and a foliation $\cal G$ on $Y$, birationally
equivalent to
$\cal F$ such that either $K_{\cal G}$ is nef, or it is a $\bb P^1$-bundle
over a curve. 
\end{theorem}

Observe that we can make the following modifications, implicit in
\cite{McQuillan08}, see for example \cite[Proposition III.2.1.]{McQuillan08}:

\begin{corollary}
\label{relativesurfaceMMP}
Let $f:X \rightarrow U$ be a projective birational morphism of surfaces, and let
$\cal F$ be foliation on $X$.
Suppose $X$ is smooth and $\cal F$ has canonical singularities.
Let $\Delta$ be a divisor
not containing any fibres of $f$.  Then we can run the relative MMP, i.e.
there is a birational map $g:X \rightarrow Y$
and $h:Y \rightarrow U$ and a foliation $\cal G$ on $Y$ such that
$K_{\cal G}+g_*\Delta$ is $h-$nef and $f = h \circ g$.
\end{corollary}
\begin{proof}
If $C$ is a $(K_{\cal F}+\Delta)$-negative curve contracted by $f$,
then in fact $\Delta\cdot C \geq 0$ and so $C$ is $K_{\cal F}$-negative. 
By the cone theorem for surface foliations, \cite[Corollary II.4.3]{McQuillan08},  we see that
$C$ is an invariant rational curve, and following \cite[III.2]{McQuillan08}
we can contract it to a point.  Notice that the contracted space still maps down to $U$.  
Continuing inductively, and letting $(Y, \cal G)$
be the output of this MMP we see that $K_{\cal G}+\Delta_Y$ is nef over $U$
where $\Delta_Y$ is the pushforward of $\Delta$ to $Y$.
\end{proof}

\subsection{Foliated bend and break}

We recall the following theorem due to \cite{Miyaoka87},
\cite[Theorem 9.0.2]{Shepherd-Barron92} or \cite{BMc01}

\begin{theorem}
Let $(X, \cal F)$ be a normal foliated variety of dimension n, and let $H_1, ..., H_{n-1}$
be ample divisors.  Let $C$ be a general intersection of
elements $D_i \in \mid m_iH_i \mid$ where $m_i \gg 0$.  
Suppose that $C \cdot K_{\cal F}<0$.
Then if $A$ is an ample divisor
through a general point of $C$
there is a rational curve $\Sigma$ tangent to $\cal F$
with 
$$A\cdot \Sigma \leq 2n\frac{A\cdot C}{-K_{\cal F}\cdot C}.$$
\end{theorem}

We make a minor modification of a lemma due to \cite{KMM94}.

\begin{corollary}[Bend and Break]
\label{bendandbreak}
Let $X$ be a normal projective variety of dimension $n$.
Let $\cal F$ be a foliation of rank $r$ on $X$, and $\Delta \geq 0$.
Let $M$ be any nef divisor.  Suppose that there are nef
$\bb R$-divisors $D_1, ..., D_n$ such that

(1) $D_1\cdot D_2 \cdot ... \cdot D_n = 0$

(2) $-(K_{\cal F}+\Delta)\cdot D_2 \cdot ... \cdot D_n >0$

Then, through a general point of $X$ there is a rational
curve $\Sigma$ with
$D_1\cdot \Sigma = 0$ and
$$M\cdot \Sigma \leq 2n\frac{M\cdot D_2 \cdot...\cdot D_n}
{-K_{\cal F}\cdot D_2 \cdot ... \cdot D_n}$$
and $\Sigma$ is tangent to $\cal F$
\end{corollary}

\begin{proof}
We can pick ample $\bb Q$-divisors $H_2, ... H_n$ sufficiently
close to $D_2, ..., D_n$ so that
$$-K_{\cal F} \cdot H_2 \cdot...\cdot H_n > \Delta\cdot H_2 \cdot...\cdot H_n \geq 0$$
Pick $m_i\gg 0$ such that $m_iH_i$ is very ample, and let $C$
be an intersection of general elements in $\mid m_iH_i \mid$.
Then, we may take $C$ to be contained in the smooth
locus of both $X$ and $\cal F$.

Then, apply the above theorem to 
give rational curves $\Sigma_k$ tangent to the foliation with

\begin{align*}
(kD_1+H)\cdot \Sigma_k \leq 2n\frac{(kD_1+H)\cdot m_2H_2 \cdot...\cdot m_nH_n}
{-K_{\cal F}\cdot m_2H_2 \cdot ... \cdot m_nH_n} \\
= 2n\frac{(kD_1+H)\cdot H_2 \cdot...\cdot H_n}
{-K_{\cal F}\cdot H_2 \cdot ... \cdot H_n}
\end{align*}

As $H_i$ approaches $D_i$, the right hand side of the inequality
approaches a bounded constant.
Thus, as $k$ varies, $\Sigma_k = \Sigma$ belongs to a bounded family, so for $k\gg0$
we may take $\Sigma$ to be fixed.  Letting $H$ approach $M$ and letting $k$
go to infinity gives our result.
\end{proof}

\begin{remark}
Observe that this result is totally independent of either the
rank of the foliation or the dimension of the ambient variety.
We recover the usual form of bend and break when we take the rank
of the foliation to be $r = \text{dim}(X)$.
\end{remark}

\section{Some adjunction results for foliations}
Many of the results in this section are known for rank 1 foliations on surfaces, equivalently
co-rank 1 foliations on surfaces.  The statements (and proofs) for co-rank 1 foliations in general 
are similar, but since
we could not find these results in the literature already we have decided to include
them here.

We begin with a simple lemma:

\begin{lemma}
\label{smoothdiscrep}
Let $f: Y \rightarrow X$ be a morphism of normal varieties.  Let $\cal F$ be 
a foliation on $X$.
Suppose that $f(Y)$ is not tangent to $\cal F$ and that $f(Y)$
is not contained in $\text{sing}(X)$.
Let $\cal F_Y$ be the pulled
back foliation.
Suppose $K_X+\Delta_X$ is $\bb R$-Cartier and either

(i) $N^*_{\cal F}$ is
a line bundle (e.g. $X$ is smooth)
or,

(ii) we have a morphism 
$f^*\Omega^{[1]}_X \rightarrow \Omega^{[1]}_Y$ between
sheaves of reflexive differentials, and $(N^*_{\cal F})^{**}$ is a line bundle.
Here $\Omega^{[1]}_X$ means $(\Omega^1_X)^{**}$.

Then
$K_{\cal F}+\Delta_X$ is $\bb R$-Cartier and
$$f^*(K_{\cal F}+\Delta_X) - K_{\cal F_Y} =
f^*(K_X+\Delta_X) - K_Y +\Theta$$
where $\Theta \geq 0$.
\end{lemma}
\begin{proof}
By Lemma \ref{basicadjunction}
 we have the equality $K_X = K_{\cal F}+[N^*_{\cal F}]$ and by assumption $[N^*_{\cal F}]$
is a Cartier divisor.  Thus $K_{\cal F}+\Delta_X = (K_X+\Delta_X)-[N^*_{\cal F}]$ and so $K_{\cal F}+\Delta_X$
is $\bb R$-Cartier.

It suffices to prove the stated equality outside of a codimension 2 subset on $Y$,
and so we may assume that $Y$ is smooth.

By definition in case (i) $N^*_{\cal F_Y}$ is the saturation of
$f^*N^*_{\cal F}$ in $\Omega^1_Y$ 
(in case (ii) it is the saturation of $f^*((N^*_{\cal F})^{**}))$ in $\Omega^1_Y$).
And so $f^*N^*_{\cal F} = N^*_{\cal F_Y} - \Theta$ (in case (ii)
$f^*((N^*_{\cal F})^{**})) = N^*_{\cal F_Y} - \Theta$).  Apply Lemma \ref{basicadjunction}
to conclude.
\end{proof}

\begin{remark}
Observe that if $X$ is not klt the morphism
$f^*\Omega^{[1]}_X \rightarrow \Omega^{[1]}_Y$ does
not always exist.
\end{remark}

Of particular interest are the cases where 
$f$ is a closed immersion, 
$f$ is a blow up or $f$ is a fibration.  In these cases, we get

\begin{corollary}
\label{smoothdiscrepcorollary}
Let $X$ be smooth.

(1) Let $\nu: D^{\nu} \rightarrow D \subset X$ be the normalization of a divisor transverse
to the foliation, then
$\nu^*(K_{\cal F} + D) = K_{\cal F_{D^\nu}}+\Theta$ where $\Theta \geq 0$.
Furthermore, $\nu(\Theta)$ is either contained in $\text{sing}(D)$ or is
tangent to $\cal F$.

(2) The foliation discrepancy is less than or equal to the usual discrepancy.

(3) If the fibres of $f: Y \rightarrow X$ are all reduced, then
$f^*K_{\cal F} - K_{\cal F_Y}= f^*K_X - K_Y$.
\end{corollary}
\begin{proof}
The only part which doesn't follow from Lemma \ref{smoothdiscrep}
is the claim in item (1) that $\nu(\Theta)$ is either contained in $\text{sing}(D)$
or is tangent to $\cal F$.  This is a local problem and so may be checked
in a neighborhood of a general point of $\nu(\Theta)$ where $\cal F$ is defined by a 1-form
$\omega$.

Write $\nu^*(K_X+D) = K_{D^\nu}+\Delta$ where $\Delta$ is the usual different.  Observe that
$\Delta$ is supported on $\nu^{-1}(\text{sing}(D))$.

Thus, we see that that $\Theta$ is supported on the union of the zero locus of $\nu^*\omega$
and $\text{supp}(\Delta)$.
Let $C \subset D^{\nu}$ 
be a component of the zero locus of $\nu^*\omega$ so that $\nu(C)$ is not contained
in $\text{sing}(D)$ and let $i: \nu(C) \rightarrow X$ be the inclusion.  Then observe
that $i^*\omega =0$ and so $\nu(C)$ is tangent to $\cal F$.
\end{proof}


The following is a more general version of foliation adjunction that we will need: 

\begin{proposition}
\label{foliationadjunction}
Let $\cal F$ be a co-rank 1 foliation on a normal projective variety $X$, 
let  $S$ be a prime divisor transverse
to the foliation, with normalization $S^\nu$,
and let $\cal F_{S^\nu}$ the foliation restricted to $S^\nu$.  

Let $\Delta$ be an effective divisor
such that $S$ is not contained in the support of $\Delta$.
If $K_{\cal F}+\Delta+S$ is an $\bb R$-Cartier divisor then
$$\nu^*(K_{\cal F}+\Delta+S) = K_{\cal F_{S^\nu}}+\Delta_{S^\nu}$$
where $\Delta_{S^\nu}\geq 0$.

Furthermore $\Delta_{S^\nu}$ is supported on 
$\nu^{-1}(\text{sing}(X)\cup\text{sing}(S)\cup \text{supp}(\Delta))$
and on centres tangent to the foliation.
\end{proposition}
\begin{remark}
Observe that we are not assuming that $K_X+\Delta+S$ is $\bb Q$-Cartier.
\end{remark}
\begin{proof}
We construct $\Delta_{S^\nu}$ as follows: 
pass to a log resolution $\pi:Y \rightarrow X$ of $(X, \Delta+S)$
and write $K_{\cal F_Y}+S'+\Delta'+E = \pi^*(K_{\cal F}+S+\Delta)$.  Let $\sigma:S' \rightarrow S^\nu$
factor through $S' \rightarrow S$.  We may apply Corollary \ref{smoothdiscrepcorollary} item (i) 
to $(\cal F_Y, S'+\Delta'+E)$
(since $Y$ is smooth) and push forward along $\sigma$.

It remains to check that $\Delta_{S^\nu} \geq 0$.  Using Corollary \ref{relativesurfaceMMP},
the line of argument of the general adjunction statement in the $K_X$ case
works equally well in the foliated situation, see for example \cite[9.2.1]{Ambro07}
or \cite[14.1]{Fujino09}. For the reader's convenience we explain the argument here.

By taking general hyperplane cuts we may reduce to the case where $X$ is a surface,
$\cal F$ is a foliation by curves and $S$ is a curve.
Let $\mu: Y \rightarrow X$ be a log resolution of $(X, \Delta+S)$, i.e.
$Y$ is smooth and $\text{exc}(\mu)\cup \mu_*^{-1}(\text{supp}(\Delta+S))$ is an snc divisor,
and let $\cal G$ be the transformed foliation.  Perhaps passing to a higher model
we may assume that $\cal G$ has canonical singularities.
Let $\Delta'$ denote the strict transform of $\Delta_Y$ and let $S_Y$ denote the strict transform
of $S$.  

By Corollary \ref{relativesurfaceMMP} we may run a $K_{\cal G}+S_Y$-MMP over $X$
terminating in $\pi: \widetilde{X} \rightarrow X$.  We claim that this MMP
only contracts curves $E$ which are disjoint from the strict transform of $S$.  
We prove this by induction
on the number of steps in the MMP, and so let $Y \rightarrow W$ be some intermediate step
of the MMP and let $S_W$ denote the strict transform of $S$ and $\Delta_W$ the strict transform of $\Delta$.  
Observe that by assumption $W$
is smooth in a neighborhood of $S_W$ so if $E$ is any curve meeting $S_W$ we have that $S_W\cdot E \geq 1$.
On the other hand, the foliated MMP over $X$ only contracts curves with $0> K_{\cal F_W}\cdot E \geq -1$.
Thus if $(K_{\cal F_W}+\Delta_W+S_W)\cdot E <0$ we see that $E$ must be disjoint from $S_W$.

Therefore we see that $\widetilde{X}$ is smooth in a neighborhood of $S_{\widetilde{X}}$ and moreover
$S_{\widetilde{X}}$ is also smooth.
Write
\[(K_{\cal F_{\widetilde{X}}}+\Delta_{\widetilde{X}}+S_{\widetilde{X}})+\Gamma = \pi^*(K_{\cal F}+\Delta+S)\]
where $\Gamma$ is $\pi$-exceptional.  Observe that since $K_{\cal F_{\widetilde{X}}}+S_{\widetilde{X}}$
is $\pi$-nef and since $\Delta_{\widetilde{X}}$ is an effective divisor containing no $\pi$-exceptional divisor that
\[-\Gamma = (K_{\cal F_{\widetilde{X}}}+\Delta_{\widetilde{X}}+S_{\widetilde{X}}) - \pi^*(K_{\cal F}+\Delta+S)\]
is $\pi$-nef and so the negativity lemma, \cite[Lemma 3.38]{KM98}, applies to show that $\Gamma \geq 0$.

For ease of notation set $T = S_{\widetilde{X}}$ and observe that we have an isomorphism $T \rightarrow S^\nu$.

Write $(K_{\cal F_{\widetilde{X}}}+\Delta_{\widetilde{X}}+S_{\widetilde{X}}+\Gamma)\vert_T = K_{\cal F_T}+\Delta_T$
and by the fact that $\Gamma \geq 0$ and by Corollary \ref{smoothdiscrepcorollary} we see
that $\Delta_T \geq  0$.  On the other hand, by construction we see that $\Delta_T = \Delta_{S^\nu}$ and so we are done.
\end{proof}

\begin{defn}
We will refer to $\Delta_{S^\nu}$ as the {\bf foliated different}.
\end{defn}

We also have a foliated Riemann-Hurwitz formula:

\begin{proposition}
\label{prop_riemannhurwitz}
Let $\pi: Y \rightarrow X$ be a surjective, 
finite morphism of normal varieties.  Let $\cal F$
be a co-rank 1 foliation on $X$, with $K_{\cal F}$ $\bb Q$-Cartier
and let $\cal F_Y$ be the pulled back foliation.
Then $$K_{\cal F_Y} = \pi^*K_{\cal F}+\sum \epsilon(D)(r_D-1)D$$
where the sum is over divisors $D$ with ramification index $r_D$, 
\end{proposition}
\begin{proof}
This is proven in \cite[pp. 20-21]{Brunella00} where $X, Y$ are surfaces, but the proof
works equally well for any co-rank 1 foliation.
\end{proof}

\begin{remark}
If the ramification of $\pi:(Y, \cal G) \rightarrow (X, \cal F)$ is 
foliation invariant then $K_{\cal G} = \pi^*K_{\cal F}$.
\end{remark}

Later on we will need to compute the discrepancies
of pairs $(\cal F, \Delta)$.
The following two results will be useful in this regard.

\begin{corollary}
\label{quotientsmoothdiscrep}
Suppose $X$ is klt and $\bb Q$-factorial and let $\cal F$ be a co-rank 1 foliation. 
Let $\Delta$
be an effective divisor.
Let $\pi: Y \rightarrow X$ be a birational morphism which extracts
divisors of usual discrepancy with respect to $(X, \Delta)$ $\leq -1$. 
Then if $\pi$ extracts $E$,
the discrepancy of $E$ with respect to $(\cal F, \Delta)$ is $\leq -\epsilon(E)$
with strict inequality if $\epsilon(E) = 0$. In particular,
$\pi$ only extracts divisors of foliation discrepancy $<0$.
\end{corollary}

\begin{remark}
This result can be phrased as saying that the non-klt
places of $(X, \Delta)$ are non-klt places of $(\cal F, \Delta)$.
\end{remark}

\begin{proof}
The statement can be checked locally on $X$, so consider
the following diagram:
\begin{center}
\begin{tikzcd}
Y' \arrow{r}{g} \arrow{d}{\pi'} & Y \arrow{d}{\pi} \\
X' \arrow{r}{f} & X
\end{tikzcd}
\end{center}
Here $f: X' \rightarrow X$
is the index 1 cover associated to $N^*_{\cal F}$, note $f$ is \'etale
in codimension 2.
Denote by $\cal F'$ the foliation on $X'$ and let
$Y'$ be the normalization
of $X' \times_X Y$. Observe that $g$ is finite.

Next, note $f^*K_X = K_{X'}$ and $f^*K_{\cal F} = K_{\cal F'}$.
Write $\Delta' = f^*\Delta$

Let $E$ be a divisor contracted by $\pi$ and let $E'$ be a divisor contracted
by $\pi'$ such that $g(E') = E$, let $r$ be the ramification index.
Working around a general point of $E, E'$ we may 
write
\begin{align*}
K_Y+\pi_*^{-1}\Delta = \pi^*(K_X+\Delta)+aE\\
K_{\cal F_Y}+\pi_*^{-1}\Delta = \pi^*(K_{\cal F}+\Delta)+bE\\
K_{Y'}+\pi'^{-1}_*\Delta' = \pi'^*(K_{X'}+\Delta') +a'E'\\
K_{\cal F_{Y'}}+\pi'^{-1}_*\Delta' = \pi'^*(K_{\cal F'}+\Delta')+b'E'.
\end{align*}

We have $(N^*_{\cal F'})^{**}$
is a line bundle sub-sheaf of $\Omega^{[1]}_{X'}$.
Next, by \cite[Theorem 4.3]{GKKP11} we have 
a morphism 
$$d\pi': \pi'^{[*]}(N^*_{\cal F'})^{**} \rightarrow \Omega^{[1]}_{Y'}.$$
We apply Lemma \ref{smoothdiscrep} to see $b' \leq a'$.

Next, by Riemann-Hurwitz,
\begin{align*}
K_{Y'}+\pi'^{-1}_*\Delta' = 
g^*(K_Y+\pi_*^{-1}\Delta)+(r-1)E' = \\
g^*(\pi^*(K_X+\Delta)+aE)+(r-1)E'.
\end{align*}
Pulling back the other way around the diagram
shows that
$$a' = ra+(r-1).$$
Likewise, foliated Riemann-Hurwitz, Proposition \ref{prop_riemannhurwitz},
 tells us that
\begin{align*}
K_{\cal F_{Y'}}+\pi'^{-1}_*\Delta' = 
g^*(K_{\cal F_Y}+\pi_*^{-1})+\epsilon (r-1)E' =\\
g^*(\pi^*(K_{\cal F}+\Delta)+bE)+\epsilon(r-1)E'
\end{align*}
where $\epsilon = 0$ if $E'$ is invariant and $ = 1$ otherwise.
Again, pulling back the other way around the diagram gives
$$b' = rb+\epsilon(E)(r-1).$$ 
Since $a \leq -1$, we get that $a' \leq -1$.
And so $rb+\epsilon(E)(r-1) = b' \leq a' \leq -1$.
This gives that $b \leq \frac{-\epsilon(E)(r-1)-1}{r} \leq -\epsilon(E)$
with strict inequality if $\epsilon(E) = 0$.
\end{proof}

\begin{lemma}
\label{transdiscrep}
Let $\cal F$ be a co-rank 1 foliation with non-dicritical singularities on a $\bb Q$-factorial 
threefold $X$. 
Let $\pi:X' \rightarrow X$ be a birational morphism.  Suppose $Z$
is a centre transverse to the foliation in the sense of Definition \ref{tangtransdef}. 
Then the foliation discrepancy
of a divisor $E$ centred over $Z$ is equal to the usual discrepancy. 
\end{lemma}
\begin{proof}
Perhaps passing to a higher resolution $\mu: X'' \rightarrow X'$ we 
may assume that $(X'', \text{exc}(\pi \circ \mu))$ is log smooth. Let $\cal G$
be the foliation pulled back to $X''$.  Perhaps shrinking around a general point
of $Z$ we may assume that every exceptional divisor dominates $Z$.  Let $E$ be one such divisor.

By non-dicriticality of $\cal F$ we see that if $\cal H$ is the foliation restricted
to $E$ then $\cal H$ is induced by the fibration $\sigma: E \rightarrow Z$.
Let $f$ be a general fibre of $\sigma$.  Then notice
that $K_{\cal H}\cdot f = K_E\cdot f$ then
\[(K_{\cal G}+E) \cdot f = K_{\cal H}\cdot f = K_E\cdot f =  (K_{X''}+E)\cdot f\]
where the first equality follows from foliation adjunction, Corollary \ref{smoothdiscrepcorollary}, and the third equality follows from usual
adjunction, hence $K_{\cal G}\cdot f = K_{X''}\cdot f$.

Repeating this computation for every exceptional divisor $E$
shows (perhaps shrinking $X$ a bit more) that $K_{\cal G}$ and $K_{X''}$ are $f$-numerically
equivalent, and so their discrepancies agree.
\end{proof}

\section{Foliation sub-adjunction}

In this section we prove a foliated version of sub-adjunction.
We recall the definition of dlt and some related results:

\begin{defn}
A pair $(X, \Delta = \sum a_i\Delta_i)$ is called {\bf divisorial log terminal (dlt)}
if $0 \leq a_i \leq 1$ and 
there exists a log resolution $\pi: (Y, \Gamma) \rightarrow (X, \Delta)$
such that $\pi$ only extracts divisors of discrepancy $>-1$.
\end{defn}

We will need the following result due to Hacon on the existence of dlt models,
see for example \cite[Theorem 10.4]{Fujino09}:

\begin{theorem}
Let $X$ be a quasi projective variety, and $B$ a boundary such that
$K_X+B$ is $\bb R$-Cartier.
One can construct a projective birational morphism $f:Y \rightarrow X$
where $Y$ is normal and $\bb Q$-factorial.
Furthermore, $f$ only extracts divisors of discrepancy $\leq -1$,
and if we set $B_Y = f_*^{-1}B + \sum_{f-exceptional} E$,
then $(Y, B_Y)$ is dlt.
\end{theorem}

\begin{lemma}
Let $(X, \Delta)$ be dlt and let $S_1, ..., S_k$ be the irreducible
components of $\lfloor \Delta \rfloor$.

(1) $(X, \Delta)$ is log canonical. 

(2) $S_i$ is normal and if we write $(K_X+\Delta)\vert_{S_i} = K_{S_i}+\Delta_i$
then $(S_i, \Delta_i)$ is dlt.

(3) If $\lfloor \Delta \rfloor = 0$ then $(X, \Delta)$ is klt
\end{lemma}
\begin{proof}
Standard, see for example \cite{KM98}.
\end{proof}

\begin{defn}
Given a pair $(X, \Delta)$ or $(\cal F, \Delta)$ we say that $W$
is a {\bf log canonical centre} of $(X \text{ or } \cal F, \Delta)$ 
if $(X \text{ or } \cal F, \Delta)$ is log canonical above the generic point of $W$,
and there is a divisor $D$ of discrepancy $=-\epsilon(D)$ dominating $W$.
\end{defn}

\begin{theorem}
\label{foliatedsubadjunction}
Let $X$ be a $\bb Q$-factorial and klt threefold and let $\cal F$ be a co-rank 1 foliation
with non-dicritical singularities.
Suppose that $W$ is a log canonical centre of $(\cal F, \Delta)$ with $\text{dim}(W) = 1$
Furthermore, suppose that $W$ is transverse to the foliation in the sense of Definition \ref{tangtransdef}.

Then $(K_{\cal F}+\Delta)\cdot W \geq 0$.
\end{theorem}
\begin{proof}
Let $\nu:W^\nu \rightarrow W$ be the normalization
and let $G = \nu^*(K_{\cal F}+\Delta)$.

First, notice that since $W$ is
transverse if $E$ is any divisor such
that the centre of $E$ on $X$ is $W$ then $E$ is transverse to the foliation and we have by Lemma \ref{transdiscrep}
that $a(E, \cal F, \Delta) = a(E, X, \Delta)$.  In particular,
$W$ is a log canonical centre of $(X, \Delta)$.

Let $f:(Y, \cal H) \rightarrow (X, \cal F)$ be a dlt modification of $(X, \Delta)$
and write $$f^*(K_X+\Delta) = K_Y+\Gamma'$$
and $$f^*(K_{\cal F}+\Delta) = K_{\cal H}+\Gamma.$$
Since $f$ only extracts divisors of usual discrepancy $\leq -1$,
by Corollary \ref{quotientsmoothdiscrep} 
it only extracts divisors of foliation discrepancy $\leq 0$, 
and so $\Gamma \geq 0$.

Observe that if $E$ is a divisor then the coefficient of $E$ in $\Gamma$ (respectively $\Gamma'$)
is $-a(E, \cal F, \Delta)$ (respectively $-a(E, X, \Delta)$) and so by Lemma \ref{transdiscrep} 
we know that $\Gamma$ and $\Gamma'$ agree above the generic point of $W$.

Let $E \rightarrow W$ be a divisor dominating $W$ which
has coefficient 1 in $\Gamma$.
$E$ is transverse to $\cal H$ and if we write
$\cal H_E$ for the foliation restricted to $E$ we have that $\cal H_E$
is the foliation induced by $\sigma:E \rightarrow W^\nu$.

Write $(K_Y+\Gamma')\vert_E = K_E+\Theta'$ and 
$(K_{\cal H}+\Gamma)\vert_E = K_{\cal H_E}+\Theta$.  Note that $\Theta, \Theta' \geq 0$ and
that if $D \subset E$ is a divisor dominating $W$ it is transverse to the foliation and therefore
has the same coefficient in $\Theta$ and  $\Theta'$ (by the construction of the different, Proposition \ref{foliationadjunction},
and Lemma \ref{transdiscrep}).
Thus $(\cal H_E, \Theta)$ is lc above the generic point of $W^\nu$.

Suppose for sake of contradiction that $\text{deg}(G)<0$.  Then $K_{\cal H_E}+\Theta = \sigma^*G$
is not nef and we can apply 
the cone theorem for surface foliations, see Theorem \ref{conetheoremsurfaces} below, 
to conclude
that there is some rational curve tangent to $\cal H_E$ which is $K_{\cal H_E}+\Theta$-negative.
However, if $C$ is any curve tangent to $\cal H_E$, then
$$(K_{\cal H_E}+\Theta)\cdot C = \sigma^*G \cdot C = 0$$
a contradiction.
\end{proof}

\begin{remark}
This should be viewed as a foliated version of Kawamata's subadjunction, \cite{Kawamata98}.  
Indeed, with
more work it is possible to prove foliated subadjunction for $X$ and $W$ of any dimension.  We will
only need the case where $\text{dim}(X) = 3$ and $\text{dim}(W)=1$ to prove our main result
and so have restricted our attention to this case.
\end{remark}

\begin{lemma}
\label{negativecurvetransverse}
Let $X$ be a
$\bb Q$-factorial and klt threefold, $\cal F$
a co-rank 1 foliation with non-dicritical singularities,
and let $S$ a surface transverse to the foliation.
Suppose that $(\cal F, \Delta)$ has canonical singularities. Let 
$R$ be an extremal ray of $\overline{NE}(X)$ such that 
$K_{\cal F}+\Delta$ and $S$ are negative on $R$.

Then $R$ is spanned by the class of a curve
$C$ which is tangent to the foliation.
\end{lemma}
\begin{proof}
Let $\nu: S^\nu \rightarrow S$ be the normalization of $S$.
Write $\nu^*(K_{\cal F}+\Delta+S) = K_{\cal F_{S^\nu}}+\Theta$.
A straightforward computation (see for example \cite[Propostion 5.46]{KM98})
shows that the non-log canonical centres of $(\cal F_{S^\nu}, \Theta)$ are contained in
non-log canonical centres of $(\cal F, \Delta+S)$.

Since $R\cdot S<0$, there exists an extremal ray $R'$ in $\overline{NE}(S)$ 
such that $\nu_*R' = R$ in $\overline{NE}(X)$.

By the cone theorem for surface foliations, see Theorem \ref{conetheoremsurfaces} below, 
$R'$ is spanned by a curve $C$
and so $R$ is spanned by $\nu(C)$.
Furthermore either
\begin{enumerate}
\item $C$ is transverse to $\cal F_{S^\nu}$ and 
contained in the non-log canonical locus of $(\cal F_{S^\nu}, \Theta)$ or
\item $C$ is tangent to $\cal F_{S^\nu}$.
\end{enumerate}

In the first case observe that $\nu(C)$ is a
non-log canonical centre of $(\cal F, \Delta+S)$ transverse to $\cal F$.

Thus we may find $0<\lambda \leq 1$ so that
$\nu(C)$ is a log canonical centre of $K_{\cal F}+\Delta+\lambda S$ transverse to $\cal F$.
In fact, since $\nu(C)$ is not a log canonical centre of $(\cal F, \Delta+S)$ we see that $\lambda <1$.
However, $(K_{\cal F}+\Delta+\lambda S)\cdot C <0$, a contradiction 
of Theorem \ref{foliatedsubadjunction}.  Thus $C$ is tangent to the foliation
and so $\nu(C)$ is tangent to $\cal F$.
\end{proof}

\section{$K_{\cal F}$-negative curves tangent to $\cal F$}

Throughout this section we will assume $X$ to be a threefold and $\cal F$ a
co-rank 1 foliation.
The object of this section is to show that if there exists a $K_{\cal F}$-negative curve
tangent to the foliation, then there exists a rational $K_{\cal F}$-negative curve tangent
to the foliation.

\subsection{Existence of germs of leaves}

Cano and Cerveau in \cite{CC92} prove the following:
\begin{theorem}
Let $(0 \in X, \cal F)$ be the germ of a 3-dimensional complex manifold
with a co-rank 1 foliation.  Suppose that $\cal F$ has non-dicritical singularities.
Let $\gamma$ be a curve tangent to the foliation and not contained in 
$\text{sing}(\cal F)$.  Then $\gamma$ is contained in a unique convergent separatrix.
\end{theorem}

As noted earlier 
there are examples of dicritical singularities with no separatrices, convergent or formal.

In what follows we adapt the techniques and ideas in \cite{CC92} to work in the setting
where $X$ is singular.

We will use the following fact about simple foliation singularities
found in \cite[Proposition II.5.5]{CC92}:

\begin{lemma}
\label{simpleextension}
Let $(0 \in X, \cal F)$ be a foliated germ with simple foliation singularities.
Let $Q_i \in \text{sing}(\cal F)$ and $Q_i \rightarrow 0$.  
Suppose that at each $Q_i$ there is a convergent germ
of a separatrix $S_{Q_i}$ such that the $S_{Q_i}$ agree
on overlaps.  Then there is a convergent germ of a separatrix at $0$
which extends the $S_{Q_i}$.
\end{lemma}

\begin{lemma}
\label{l_firstextensionlemma}
Let $X$ be smooth threefold, $\cal F$ a co-rank 1 foliation with simple singularities
and $E$ a compact $\cal F$-invariant divisor.
Let $\gamma$ be a germ of a curve tangent to $\cal F$ meeting $E$
but not contained in $E$.
Then there exists a neighborhood $U$ of $E$ and a closed $\cal F$-invariant
hypersurface $S \subset U$ such that $\gamma \subset S$.
\end{lemma}
\begin{proof}
The proof is essentially a small generalization of the proofs of 
\cite[Lemma IV.1.4, Corollary IV.1.5]{CC92}.

Since $\gamma$ is not contained in $E$ we have $\gamma \cap E = \{P_1, ..., P_n\}$.
Moreover, if $U$ is a sufficiently small neighborhood of $E$ we see
that $U\cap \gamma = \bigcup_i \gamma_i$ where $\gamma_i\cap E = P_i$ and $\gamma_i = \bigcup_j\gamma_{ij}$ 
where $\gamma_{ij}$ is irreducible.
It suffices to construct a separatrix containing each $\gamma_{ij}$ and thus shrinking $X$ about
$E$ we are free to assume that $\gamma$ is irreducible and $\gamma \cap E = P$ is a single point.

Furthermore, possibly further shrinking $X$ around $E$, by passing to a resolution, $\pi:X' \rightarrow X$,
setting $E' = \pi^{-1}E$ and setting $\cal F'$ for the transformed foliation,
we may assume that each point of $\text{sing}(\cal F')\cap E'$ has at most 1
(formal) separatrix not contained in $E'\cup \text{exc}(\pi)$.

Let $W=\text{sing}(\cal F')\cap E'$, let $V$ be the union of those irreducible components $W'$ of $W$
such that at a general point of $W'$ there is exactly one separatrix not contained in $E' \cup \text{exc}(\pi)$.
Let $\gamma'$ be the strict transform of $\gamma$ and let $\gamma' \cap E' = P'$.

Observe that $P' \in \text{sing}(\cal F')\cap E'$ and there is a separatrix of $\cal F'$, call it $\Sigma$,
in a neighborhood of $P'$
containing $\gamma'$.  Since $\gamma'$ is not contained in $E'\cup \text{exc}(\pi)$ we see
that $\Sigma$ is not contained in $E' \cup \text{exc}(\pi)$ and so $P' \in V$.
Let $V_0$ be the connected component of $V$ containing $\gamma'\cap E'$. 

We let $\cal A$ be the locus of points  $Q \in V_0$
such that 

\begin{enumerate}
\item \label{i_a} there exists an open set $Q \in U_Q$ and separatrix $Q \in S'_Q \subset U_Q$
such that $S'_Q$ is not contained in $E'\cup \text{exc}(\pi)$; and

\item \label{i_b} for every $R \in E' - V_0$ there exists an open set $W_R$ containing $R$
such that $W_R \cap S'_Q = \emptyset$.  
\end{enumerate}

Observe that if $R \in E' - V_0$ and if $Q \in V_0$ and $U_Q$ is a small
open neighborhood of $Q$ and $S'_Q \subset U_Q$ is a separatrix at $Q$ then
$S'_Q \cap E' \subset V_0$ by definition and so $R \notin S'_Q$.

Thus, it is easy to see that $\cal A$ is open in $V_0$.  To see that $\cal A$ is closed in $V_0$
let $Q_i$ be a sequence of points in
$\cal A$  converging to $Q \in V_0$. By Lemma \ref{simpleextension} we can find a separatrix $S'_Q$ at $Q$
agreeing with the $S'_{Q_i}$ on overlaps, and therefore $S'_Q$ 
satisfies \ref{i_a} and \ref{i_b}.  
By assumption $\gamma' \cap E' \in \cal A$ is nonempty, and so $\cal A = V_0$.

There exist finitely many $P_i$ such that $U_{P_i}$ cover $V_0$.
For all $R \in E' - V_0$ there exists an open $W_R$
disjoint from all the $S'_{P_i}$.  
The collection $\{W_R, U_{P_i}\}$ forms an open cover of a neighborhood of $E'$, call it $U$.
Let $S' = \cup S'_{P_i}$. Note that these separatrices agree on overlaps since 
if $Q \in U_{P_i}\cap U_{P_j}\cap V_0$ there is at most 1 separatrix at $Q$ not
contained in $E'\cup \text{exc}(\pi)$ and $S'_{P_i}, S'_{P_j}$ 
are both separatrices at $Q$ not contained
in $E'\cup \text{exc}(\pi)$.  
Thus $S'$ is a well-defined closed foliation invariant
hypersurface in $U$.

Finally, since $\pi$ is proper, by the proper mapping
theorem we have that $V = \pi(U)$ is  neighborhood of $E$ and
$S = \pi(S') \subset V$ is a $\cal F$-invariant hypersurface in $V$.
\end{proof}

\begin{corollary}
\label{constructingsurfacegerm}
Let $C$ be a compact curve tangent to a co-rank 1
foliation with non-dicritical singularities on threefold $X$ such that $C$ is
not contained in $\text{sing}(\cal F)\cup\text{sing}(X)$.
Then there is a germ of an analytic surface containing $C$, call it $S$,
such that $S$
is tangent to the foliation.
\end{corollary}
\begin{proof}
First, pick a point $p \in C$ which is smooth point of $X, \cal F$ and $C$.
Since $\cal F$ has rank 2 there is a smooth germ of a curve $\gamma$ tangent to
$\cal F$ meeting $C$ at exactly $p$.

Let $\pi:(Y, \cal G) \rightarrow (X, \cal F)$ be a resolution of singularities
of $X$ and $\cal F$ so that $Y$ is smooth and $\cal G$ has simple singularties.  Passing to a higher resolution if needed we may assume
that $\pi^{-1}(C)$ is a snc divisor, which must be $\cal G$ invariant.

Let $\gamma'$ be the strict transform of $\gamma$.  By Lemma \ref{l_firstextensionlemma}
 we may find 
$S'$, the germ
of an invariant hypersurface containing $\gamma'$ in a small neighborhood $U$ of $\pi^{-1}(C)$.
By the proper mapping theorem $\pi(S') = S$ is a $\cal F$-invariant hypersurface
containing $C$.
\end{proof}

\begin{remark}
Observe that in contrast to the smooth case where every non-dicritical
singularity admits at least one convergent separatrix, if $X$
is singular it is possible for there to be no separatrices (formal or otherwise)
through a particular point $x \in X$.
In the case of surfaces an example is given by considering the contraction
of an elliptic Gorenstein leaf.
In these cases, however, there are no germs of curves tangent to $\cal F$
passing through $x$.
Intriguingly, these very same counter-examples are also intimately related to counter-examples
to abundance for foliations.
\end{remark}

\begin{corollary}
\label{constructingsurfacegermII}
Let $X$ be a threefold, and suppose $\cal F$ is a co-rank 1 foliation on $X$ 
with canonical and non-dicritical singularites
and suppose that $C \subset \text{sing}(\cal F)$ is a compact curve that is not contained in
$\text{sing}(X)$.  Furthermore, 
if $\cal F$ is simple at the generic point of $C$
then we may choose $S$ so that $S$ agrees with the strong separatrix along $C$.
\end{corollary}
\begin{proof}
Let $H$ be a general hyperplane cut of $X$ passing through $C$.  Let $P = H \cap C$ and let $\cal G$
be the foliation restricted to $H$.  By a result of Camacho and Sad, \cite[Theorem 3.3]{Brunella00},
there exists a germ of a curve, $\gamma$ tangent to $\cal G$ passing through $P$.
Moreover, if $\cal F$ has simple singularities at the generic point of $C$ then we may take $\gamma$
to be the strong separatrix at $P$.  Pushing forward, we see that $\gamma$ is a germ of a curve in $X$
tangent to $\cal F$ and meeting $C$.

As above, let $\pi:(X', \cal F') \rightarrow (X, \cal F)$ be a resolution
of singularities of $X$ and $\cal F$ such that $\pi^{-1}(C)$ is a snc divisor.
Since $C$ is contained in the singular locus, we see that $C$ is tangent
to $\cal F$ and since $\cal F$ is non-dicritical every divisor on $X'$ dominating $C$ is invariant, see
Definition \ref{tangtransdef}.
Thus $\pi^{-1}(C)$ is invariant.

By Lemma \ref{l_firstextensionlemma}there exists an extension of $\gamma$ to an invariant hypersurface
$S'$ in a neighborhood of $\pi^{-1}(C)$.  Taking $\pi(S')$ gives our desired separatrix.
\end{proof}

\begin{corollary}
\label{constructingsurfacegermIII}
Let $X$ be a threefold and
suppose that $\cal F$ is a co-rank 1 foliation with canonical and non-dicritical singularities
and suppose that $C \subset \text{sing}(X)$ is a compact curve tangent to $\cal F$.  Suppose that
$\cal F$ is canonical at the generic point of $C$.  Then there exists a germ
of a separatrix $S$ containing $C$.
\end{corollary}
\begin{proof}
Since $\cal F$ is canonical, taking a general hyperplane $H$ section
meeting $C$ we have if $P = C\cap H$ that $\cal F_H$ is canonical near $P$.
Applying the classification of canonical foliation singularities on surfaces,
see \cite[Fact I.2.4]{McQuillan08}, there exists a germ of a curve $\gamma$
tangent to $\cal F$ meeting $C$ transversely.

As above let $\pi:(X', \cal F') \rightarrow (X, \cal F)$ be a resolution of singularities 
of $X$ and $\cal F$ and so that $\pi^{-1}(C)$ is a snc divisor.  Since $C$ is tangent to $\cal F$, 
$\pi^{-1}(C)$ is invariant, and arguing as in Corollary \ref{constructingsurfacegerm} we may produce 
our desired separatrix containing $\gamma$ and
 our result follows.
\end{proof}


\subsection{Producing rational curves}
We begin here with an algebraicity criterion:

\begin{lemma}
\label{algebraicorrational}
Let $C$ be a compact curve, and $S$ an analytic surface germ containing $C$, sitting
inside a projective variety $X$ and suppose that $C$ is not contained in $\text{sing}(S)$.
Let $\nu: S^\nu \rightarrow S$ be the normalization of $S$ and let $C' \subset S^\nu$ be the strict transform
of $C$.

Assume that $K_{S^\nu}+\Delta$ is $\bb Q$-Cartier and 
$(K_{S^\nu}+\Delta)\cdot C'<0$, and that $\Delta \geq 0$ is a boundary along $C'$.
Then either $C$ is rational and $(K_{S^\nu}+\Delta)\cdot C \geq -2$
or $S$ is algebraic, i.e. the Zariski closure
of $S$ is an algebraic surface.
\end{lemma}
\begin{proof}
Let $Y$ be the Zariski closure of $S$ with $K(Y)$ the field
of rational functions on $Y$.  
The algebraicity of $S$ follows if the transcendence degree of $K(Y)$
over $\bb C$ is 2.

Let $T \xrightarrow{f} S^\nu \xrightarrow{g} S$ 
be the minimal resolution of the normalization of $S$
(perhaps after restricting $S$ to a smaller neighborhood of $C$) and let $B \subset T$
be the strict transform of $C$.
We have $K_T +\Delta_T = f^*(K_{S^\nu}+\Delta)$ where $\Delta_T \geq 0$.

Let $\mathfrak{T}$ denote the formal scheme given by the completion of
$T$ along $B$ and let $K(\mathfrak{T})$ be the field of
formal meromorphic functions on $\mathfrak{T}$.  Notice
that $K(\mathfrak{T})$ is a field extension of $K(Y)$, and
so it suffices to bound the transcendence degree of $K(\mathfrak{T})$.
Let $n: B' \rightarrow B$ be the normalization of $B$.

By assumption there exists a $t \geq 0$ such that 
$K_T+\Delta_T+tB = K_T+\Theta+B$ where $\Theta \geq 0$ and $B$ is not contained in $\text{supp}(\Theta)$.
By adjunction $(K_T+\Theta+B)\cdot B = 2g(B') -2+d_{B}$,
where $d_{B} \geq 0$.

If $n^*\cal O(B)$ is not ample,
then the left hand side of the equation is negative, hence $B$ is rational, 
and $(K_T+\Delta_T) \cdot B \geq (K_T+\Theta+B)\cdot B \geq -2$.

On the other hand, if $n^*\cal O(B)$ is ample, the normal bundle
of $B$ in $\mathfrak{T}$ is ample,
which by a result of Hartshorne,
\cite[Theorem 6.7]{Hartshorne68}, or by an observation of 
Bogomolov and McQuillan, \cite[Fact 2.1.1]{BMc01},
implies that $K(\mathfrak{T})$ has transcendence degree at most 2
over $\bb C$, and our result follows.
\end{proof}

As mentioned in the introduction this has the immediate consequence:

\begin{corollary}
Let $\cal F$ be a smooth rank 2 foliation on a smooth projective variety.
Let $C$ be a $K_{\cal F}$-negative curve tangent to $\cal F$.
Let $S$ be a germ of a leaf around $C$.  Then either
\begin{enumerate}
\item $C$ is rational or

\item $S$ is algebraic (i.e. its Zariski closure is a surface).
\end{enumerate}
\end{corollary}

In the following  proof we will
make use of the following definition:

\begin{defn}
Given a reflexive sheaf $L$ and a positive integer $q \leq \text{dim}(X)$
a {\bf Pfaff field} of rank $q$ is a non-zero morphism $\Omega^q_X \rightarrow L$.
Given a foliation $\cal F$ of rank $q$, by taking the $q$-th wedge
power of $\Omega^1_X \rightarrow \cal F^*$ we get a Pfaff field
$\Omega^q_X \rightarrow \cal O(K_{\cal F})$ of rank $q$.
\end{defn}

\begin{lemma}
\label{l_summarylemma}
Let $X$ be 3-fold.  Suppose that $K_{\cal F}$ is $\bb Q$-Cartier
and $\cal F$ has only non-dicritical singularities.

Let $C$ be a compact curve tangent to the foliation such that either $C$
is not contained in $\text{sing}(X) \cup \text{sing}(\cal F)$
or $\cal F$ is canonical at the generic point of $C$.

Then there exists a germ of an analytic surface $S$
such that $C$ is contained
in $S$, and $S$ is foliation invariant.

If $\nu: S^\nu \rightarrow S$ is the normalization,
then $\nu^*K_{\cal F} = K_{S^\nu}+\Delta$ where $\Delta \geq 0$.
\end{lemma}
\begin{proof}
By Lemmas \ref{constructingsurfacegerm}, \ref{constructingsurfacegermII}, \ref{constructingsurfacegermIII} 
we get the existence of the germ $S$
containing $C$.

To prove our last statement,
if $\Omega_X^2 \rightarrow \cal O(K_{\cal F})$ is the Pfaff field associated
to our foliation, since $S$ is foliation invariant we
have a morphism 
$(\Omega_S^2)^{\otimes m} \rightarrow \cal O_S(mK_{\cal F})$,
where $m$ is the Cartier index of $K_{\cal F}$.
We can apply \cite[Lemma 3.7]{AD14}
to see that this lifts to a map
$(\Omega_{S^\nu}^2)^{\otimes m} \rightarrow \nu^*O_S(mK_{\cal F})$.
Observe that the lemma is proven in the case
where $S$ is a variety, however
the proof works just as well in the case where $S$ is an analytic variety.

Thus, we have a nonzero map
$\cal O(mK_{S^\nu}) \rightarrow \cal O(m\nu^*K_{\cal F})$
and our result follows.
Observe that $\Delta$ is supported on the locus
where this map fails to be surjective, which is contained
within $\text{sing}(X) \cup \text{sing}(\cal F)$.
\end{proof}

\begin{example}
\label{explicitcomputation}
In the case that $X$ is smooth with simple singularities, the computation
of $\Delta$ is easy.  $\Delta$ is supported on $\nu^{-1}(\text{sing}(\cal F))$
and if $Z$ is a component of $\text{sing}(\cal F)$ and $S$ is a strong
separatrix along $Z$, the coefficient of $Z$ in $\Delta$ is exactly 1.  Otherwise
the coefficient of $Z$ is some positive integer $k$ which depends on the analytic
type of the singularity.

More generally, passing to a resolution $\pi:(X', \cal F') \rightarrow (X, \cal F)$
of $X, S$ and so that $\cal F'$ has simple singularities we can write
$$K_{\cal F'}+\Gamma = \pi^*K_{\cal F}$$
and 
$$(K_{\cal F'}+\Gamma)\vert_{S'} = K_{S'}+\Delta'$$
where $\Delta'$ can be computed as above.  If $\sigma:S' \rightarrow S^\nu$ is the induced morphism
we have that $\Delta = \sigma_*\Delta'$.
\end{example}

\begin{defn}
McQuillan's classification of $\bb Q$-Gorenstein canonical surface foliation singularities,
\cite[Fact I.2.4]{McQuillan08},
implies that the underlying surface has at worst quotient singularities.  Thus if $(\cal F, \Delta)$
has canonical singularities, we see that $X$ has at worst quotient singularities
in codimension 2.
In this situation, if $C \subset \text{sing}(X)$ we will say $C$ is contained in $\text{sing}(\cal F)$
if around a general point of $C$ there exists a quotient $q:(X', \cal F') \rightarrow (X, \cal F)$
with $X'$ smooth
so that $q^{-1}(C) \subset \text{sing}(\cal F')$.
\end{defn}

We finish the section with our characterization of $(K_{\cal F}+\Delta)$-negative
curves tangent to $\cal F$.

\begin{lemma}
\label{negativecurvestangenttofoliation}
Let $X$ be a projective $\bb Q$-factorial threefold and let $\cal F$ be a co-rank 1 foliation
on $X$.

Let $C$ be a curve tangent to $\cal F$,
with $(K_{\cal F}+\Delta)\cdot C<0$.  Suppose that $\cal F$
has non-dicritical singularities and that $(\cal F, \Delta)$ is canonical.
Then,
$[C]= \sum a_i [M_i]+\beta$
where $(K_{\cal F}+\Delta)\cdot \beta \geq 0$ and the $M_i$ are 
rational curves tangent to the foliation with
$0 > K_{\cal F}\cdot M_i \geq -4$.
\end{lemma}
\begin{proof}
Let $S$ be the germ of a surface tangent to the foliation
containing $C$ whose existence is guaranteed by Lemma \ref{l_summarylemma}
 and let $\nu:S^\nu \rightarrow S$ be the normalization.  
Again, by Lemma \ref{l_summarylemma}
 we may write $\nu^*(K_{\cal F}+\Delta) = K_{S^\nu}+\Theta$.

First, suppose that $C \subset \text{sing}(\cal F)$. 
By McQuillan's classification of canonical surface singularities, 
around a general point of $C$ 
we have a quotient $q:(X', \cal F') \rightarrow (X, \cal F)$ where $X'$ is smooth. Set
$C' = q^{-1}(C)$. We can find a smooth separatrix around a general point of  $C'$, call it $S'$, so that
$K_{\cal F'}\vert_{S'} = K_{S'}+C+\Theta'$
where $C'$ is not contained in $\text{supp}(\Theta')$. We can therefore choose our separatrix $S$
around $C$ to be such that $S = q(S')$.
A direct computation as in Example \ref{explicitcomputation} shows that 
$\Theta = C+\Theta_0$ where $C$ is not contained in $\text{supp}(\Theta_0)$.
Adjunction tells us that
\[0> (K_{S^\nu}+C+\Theta_0)\cdot C \geq 2g(C)-2\]
and so $C$ is rational.  Moreover, since $(\cal F, \Delta)$ is
canonical and $C \subset \text{sing}(\cal F)$ we see
that $C$ is not contained in $\text{supp}(\Delta)$ and so $\Delta \cdot C \geq 0$ which gives us
$(K_{\cal F}+\Delta)\cdot C \geq -2$.

Otherwise, since $(\cal F, \Delta)$ is canonical
we see that $\Theta$ is a boundary along $C$.
We can then apply our algebraicity criterion, Lemma \ref{algebraicorrational},
to see that either $C$ is rational and of bounded negativity, or $S$ is algebraic.
In the latter case we can apply the usual cone theorem for surfaces.

Computing explicitly as in Example \ref{explicitcomputation} we see
that the non-log canonical locus of $(S^\nu, \Theta)$ is supported
on the singular locus of $\cal F$ and on a finite collection of points.
The cone theorem
for surfaces tells us that in $\overline{NE}(S^\nu)$
we can write $[C] = \sum a_i[L_i] +\beta$ where the $L_i$
are curves contained in the non-log canonical locus of $(S^\nu, \Theta)$ or are rational
curves with $(K_{S^\nu}+\Theta)\cdot L_i = (K_{\cal F}+\Delta)\cdot L_i \geq -4$, 
and $(K_{\cal F}+\Delta)\cdot \beta \geq 0$. 

Again, since the non-log canonical locus
of $(S^\nu, \Theta)$ is contained in $\text{sing}(\cal F)$, if $L_i$ is $K_{\cal F}$-negative
it must be rational.
And so pushing forward to $X$ gives our result.
\end{proof}

\section{The cone theorem for surfaces}
We will need the following extension of the foliated cone theorem to foliations
with a boundary.
In proving it we
use the following definition and result from convex geometry:

\begin{defn}
Let $K$ be a convex cone containing no lines.  A ray $R$
of $K$ is called exposed if there is a hyperplane meeting $K$
exactly along $R$.
\end{defn}

\begin{lemma}
\label{exposedextremalray}
If $K$ is a closed convex cone containing no lines, then $K$ is the closure
of the subcone generated by the exposed rays.
\end{lemma}
\begin{proof}
See \cite[Corollary 18.7.1]{Rockafellar}.
\end{proof}

\begin{theorem}
\label{conetheoremsurfaces}
Let $X$ be a normal projective surface, $\cal F$ a 
rank 1 foliation,
and $\Delta = \sum a_iD_i$ an effective divisor.
Suppose that $K_{\cal F}+\Delta$ is $\bb R$-Cartier.
Then
$$\overline{NE}(X) = \overline{NE}(X)_{K_{\cal F}+\Delta \geq 0} +
Z_{-\infty} + \sum \bb R^+[L_i]$$
where $L_i$ are invariant rational curves with $(K_{\cal F}+\Delta)\cdot L_i \geq -4$, 
and $Z_{-\infty}$ is spanned by those
$D_i$ in $\text{supp}(\Delta)$ with $a_i > \epsilon(D_i)$.

In particular, if $H$ is ample, there are only finitely many curves with
extremal rays $R$ with $(K_{\cal F}+\Delta+H)\cdot R<0$.
\end{theorem}
\begin{proof}
Let $W$ denote the closure of the right hand side of the desired equality.

Assume that $W$ is strictly smaller than $\overline{NE}(X)$.
Then, by Lemma \ref{exposedextremalray}, if $H$ is a suitable ample divisor,
if we choose $t$ so that 
$H_R = K_{\cal F}+\Delta+tH$ is nef, 
it is zero precisely on one exposed
extremal ray $R$, not contained in $W$.

We argue depending on $\nu(H_R)$.  If $\nu(H_R)\leq 1$, then,
we apply foliated bend and break lemma, Corollary \ref{bendandbreak}, setting $D_i = H_R$ for $i\leq \nu(H_R)+1$
and $D_i = H$ otherwise.
Then, $D_1\cdot D_2 = 0$ and $(K_{\cal F}+\Delta)\cdot D_2 = -tH\cdot D_2 <0$.
We can then conclude that through a general point of $S$ there is a rational curve $\Sigma$ with
$D_1\cdot \Sigma = H_R\cdot \Sigma = 0$,
with \[M \cdot \Sigma \leq 4\frac{M\cdot D_2}{-K_{\cal F}\cdot D_2}\] 
where $M$ is any nef divisor.  

For $k\gg 0$ sufficiently large we see
that $A \coloneqq kH_R-(K_{\cal F}+\Delta)$ is ample
and 
\[A\cdot \Sigma \leq 4\frac{(kH_R-K_{\cal F}-\Delta)\cdot D_2}{-K_{\cal F}\cdot D_2}=
4 \frac{-(K_{\cal F}+\Delta)\cdot D_2}{-K_{\cal F}\cdot D_2} \leq 4.\]
In particular, the extremal ray
is spanned by the class $[\Sigma]$ and $A\cdot \Sigma = -(K_{\cal F}+\Delta)\cdot \Sigma \leq 4$.

If $\nu(H_R) = 2$, then, writing $H_R = A+E$ where $A$ is ample
and $E$ is effective we see that $E\cdot R<0$, and hence $R$
is spanned by some component of $E$, call it $C$.

Write $E = rC+E'$ with $r >0$.
If $\Delta$ is a boundary along $C$, we see that there exists some $\alpha \geq 0$ such
that $K_{\cal F}+\Delta+\alpha E = K_{\cal F}+\Delta'+C$ where $\Delta' \geq 0$ and $C$ is
not contained in $\text{supp}(\Delta')$.
However, we have $(K_{\cal F}+\Delta'+C)\cdot C<0$ which is a contradiction of
foliation adjunction, Proposition \ref{foliationadjunction}, if $C$ is not invariant.
Thus $C$ must be invariant and so $R$ is spanned by an
invariant curve, $C$.

If $C$ is not contained in $\Delta$ then
$\nu^*(K_{\cal F}+\Delta) = K_{C^\nu}+\Theta \geq -2$ where $\Theta \geq 0$
and where $\nu:C^\nu \rightarrow C$ is the normalization.
Thus, either $C \subset \text{supp}(\Delta)$ or $C^\nu$ is rational
and $(K_{\cal F}+\Delta)\cdot C \geq -2$.
We see then that $W$ and $\overline{NE}(X)$ coincide.

Standard arguments then apply to show that the right hand side of our
equality is already closed, and that the extremal rays are locally
discrete.
\end{proof}

\begin{remark}
Observe that $Z_{-\infty}$ is in fact the contribution to the cone coming
from the non-log canonical locus of $(X, \Delta)$.
\end{remark}

\section{The cone theorem for threefolds}
With the work of the previous sections in hand, we are now in a position
to give a proof of the foliated cone theorem.
The argument is similar to the one used to prove the cone theorem for surfaces.

\begin{theorem}
Let $X$ be a projective
$\bb Q$-factorial and klt threefold
and $\cal F$ a co-rank 1 foliation
with non-dicritical foliation singularities.  Suppose that $(\cal F, \Delta)$
has canonical singularities.
Then
$$\overline{NE}(X) = \overline{NE}(X)_{K_{\cal F}+\Delta \geq 0}+\sum \bb{R}_
+ [L_i]$$
where $L_i$ is a rational curve with 
$(K_{\cal F}+\Delta)\cdot L_i \geq -6$.

In particular, the $(K_{\cal F}+\Delta)$-negative
extremal rays are locally discrete
in the $(K_{\cal F}+\Delta)$-negative portion of the cone.
\end{theorem}

\begin{proof}
Choose a suitable  ample divisor $H$ so that $H_R = (K_{\cal F}+\Delta)+H$ is nef, and such
that $H_R$ is zero on precisely one exposed extremal ray, $R$.
We argue based on the numerical dimension of $H_R$.

If $\nu = \nu(H_R) <3$, as in the proof of Theorem \ref{conetheoremsurfaces},
we see that $R$ is spanned by a family of rational curves $\Sigma_t$
tangent to $\cal F$, passing through a general point of $X$ and  with $\Sigma_t\cdot (K_{\cal F}+\Delta) \geq -6$
 by applying Corollary \ref{bendandbreak} with $D_i = H_R$ for $1 \leq i \leq \nu(H_R)+1$
and $D_i = H$ for $\nu(H_R)+1<i \leq \text{dim}(X)=3$.

So, suppose that $\nu(H_R) = 3$, i.e. $H_R$ is big and nef.
Taking $\epsilon >0$ sufficiently small, we may take
$K_{\cal F}+\Delta+(t - \epsilon)H$ to be a big $\bb Q$-divisor 
which is negative
on $R$ and so there exists some
effective prime divisor $D$ such that $D\cdot R<0$.
We have then that
$R$ comes from an extremal ray $R' \subset \overline{NE}(D^\nu)$
where $\nu: D^\nu \rightarrow D$ is the normalization of $D$.

Suppose first that $D$ is invariant.
By Lemma \ref{l_summarylemma} we may write
 $\nu^*(K_{\cal F}+\Delta) = K_{D^\nu}+\Delta_{D^\nu}$ where $\Delta_{D^\nu} \geq 0$. The (classical) the cone theorem for surfaces
shows that $R'$ is spanned by a $K_{D^\nu}$-negative rational curve or it is spanned by
a curve contained in $\text{supp}(\Delta_{D^\nu})$.  In either case, we see that $R$ is spanned
by a curve tangent to $\cal F$ in which case may apply
Lemma \ref{negativecurvestangenttofoliation} 
to conclude.

Now suppose that $D$ is transverse to the foliation.
In this case, since $R$ is $K_{\cal F}+\Delta$ and $D$-negative
Lemma \ref{negativecurvetransverse}
applies to show that $R$ is spanned by the class
of a curve tangent to the foliation, in which case Lemma \ref{negativecurvestangenttofoliation}
applies again.

In both cases, $R$ is spanned by the class of a rational curve $C$ tangent to
the foliation with $(K_{\cal F}+\Delta) \cdot C \geq -4$.

Our result then follows by standard arguments to show
that the cone of curves indeed has the claimed structure.
\end{proof}

\section{The contraction theorem}

In this section we will prove a contraction result for $K_{\cal F}+\Delta$-negative
extremal rays in the cone of curves.

\begin{defn}
Let $R$ be a $K_{\cal F}+\Delta$-negative extremal ray in $\overline{NE}(X)$.
By a {\bf contraction} of $R$ we mean a morphism $c_R:X \rightarrow Y$ between
normal varieties with $c_{R*}\cal O_X = \cal O_Y$ and $c_R(\Sigma) = p$, a point,
if and only if $[\Sigma] \in R$.
\end{defn}

Our goal is to prove the following:

\begin{theorem}
Suppose $X$ is a projective $\bb Q$-factorial and klt threefold.  Suppose that $\cal F$
has non-dicritical singularities and that $(\cal F, \Delta)$ is
canonical and is terminal along $\text{sing}(X)$.

Let $R$ be a $K_{\cal F}+\Delta$-negative extremal ray.  Then there is a contraction
of $R$ which only contracts curves tangent to $\cal F$.
\end{theorem}

We will handle the cases of fibre, divisorial and flipping type contractions separately.

\begin{defn}
Given an extremal ray $R \subset \overline{NE}(X)$ we define
$\text{loc}(R)$ to be all those points $x$ such that there exists a curve $C$ with $x \in C$ 
and $[C] \in R$.
\end{defn}

\begin{lemma}
\label{l_closedness}
Let $X$ be a $\bb Q$-factorial threefold.
Let $R$ be a $K_{\cal F}+\Delta$-negative extremal ray.  Then $\text{loc}(R)$ is closed.
\end{lemma}
\begin{proof}
Let $H_R$ be a supporting hyperplane to $R$.

Suppose that $\nu(H_R)<3$.
As in the proof of the cone theorem, 
we see that $R$ is spanned by a family of rational curves
tangent to $\cal F$ and passing through a general point of $X$
 by applying Corollary \ref{bendandbreak} with $D_i = H_R$ for $1 \leq i \leq \nu(H_R)+1$
and $D_i = H$ for $\nu(H_R)+1<i \leq \text{dim}(X)=3$ where $H$ is an ample divisor.

Let $B$ be the closed subscheme of the Hilbert scheme of $X$ parametrizing those subvarieties
tangent to $\cal F$, see \cite[\S 2.3]{LPT11}.  We may consider the following diagram 
as in \cite[\S 4]{Kollar91c}
\begin{center}
\begin{tikzcd}
U \arrow{r}{F} \arrow{d}{p} & X\\
Z &
\end{tikzcd}
\end{center}
where $Z$ is the closed subscheme of $B$ whose general element
paramatrizes a rational curve tangent to $\cal F$ and spanning $R$.
Let 
$p: U \rightarrow Z$ be the universal family over $Z$ and $F:U \rightarrow X$ the evaluation map and observe
that $F$ is surjective.

However, observe that if $\Sigma$ is any curve contained in a fibre of $p$
then $[F(\Sigma)] \in R$
and so $\text{loc}(R) = X$.

Otherwise $H_R$ is big and nef,
and so there exists an irreducible effective divisor $S$ with 
$R\cdot S<0$, in particular $\text{loc}(R) \subset S$.
We can write $H_R = K_{\cal F}+\Delta+t\epsilon(S)S+A$ where $A$ is ample and
$t$ is the log canonical threshold of $(\cal F, \Delta)$ along $S$.
Write $G = H_R\vert_{S^\nu}$ where $S^\nu \rightarrow S$ is the normalization 
and observe that $G = K_{S^\nu} +\Theta+\nu^*A$ or $ = K_{\cal F_{S^\nu}}+\Theta+\nu^*A$
depending on whether $S^\nu$ is invariant or not and
where $\Theta \geq 0$.

If $\nu(G)=2$ then $\text{loc}(R)$ is a finite collection of curves.
Otherwise, as above, we can apply bend and break, to produce through every point $x \in S$
a curve spanning $R$ and so $\text{loc}(R) = S$.
\end{proof}

\begin{defn}
Let $\text{loc}(R) = Z$.  If $\text{dim}(Z) = 3$ we say that the contraction
corresponding to $R$ is of {\bf fibre type}, if $\text{dim}(Z) = 2$ we say that the contraction
corresponding to $R$ is of {\bf divisorial type} and if $\text{dim}(Z) = 1$ we say
that the contraction is of {\bf flipping type}.
\end{defn}

\subsection{Preliminary computations}

We collect here several computations that we will use
repeatedly through this section.

\begin{lemma}
\label{termsingtangent}
Let $X$ be a $\bb Q$-factorial and klt threefold.
Suppose that $(X, \cal F)$ has non-dicritical terminal
foliation singularities.  Then $\text{sing}(X)$ is tangent to
$\cal F$.
\end{lemma}
\begin{proof}
Suppose not.  Then there exists a curve $C \subset \text{sing}(X)$
transverse to $\cal F$.  Let $\pi:(Y, \cal G) \rightarrow (X, \cal F)$ be a resolution of
singularities of $X$.  Perhaps shrinking around a general point of $C$ we may
assume that $\pi$ only extracts divisors which dominate $C$ 
and which are transverse to the foliation.  Furthermore, we may run a $K_Y$-MMP
over $X$ and by replacing $Y$ by the output of this MMP we may 
assume that $K_Y$ is $\pi$-nef.  

On the other hand, since $\cal F$ is terminal, by the negativity lemma, \cite[Lemma 3.38]{KM98}, we know that $K_{\cal G}$
is not $\pi$-nef, so
let $f \subset Y$ be a curve mapping to a point with $K_{\cal G}\cdot f <0$, and let $E \subset Y$
be a divisor with $E\cdot f<0$.

Let $(E, \cal H)$ be the restriction of $\cal G$ to $E$.  By non-dicriticality,
$\cal H$ is induced by the fibration $E \rightarrow C$.  
By foliation adjunction, Proposition \ref{foliationadjunction}, we can
write $(K_{\cal G}+E)\vert_E = K_{\cal H}+\Delta$ where $\Delta \geq 0$.  Since $K_{\cal G}, E$ are
both Cartier and intersect $f$ negatively we get that
$-2 \leq K_{\cal H}\cdot f \leq -2$, 
and so in fact $\cal H$ is induced by a fibration in rational curves and $E\cdot f = -1$.

Thus, $K_E\cdot f = -2 = K_{\cal H}\cdot f$, which since $E\cdot f = -1$, 
implies that $K_Y\cdot f =-1$, a contradiction of the $\pi$-nefness of $K_Y$.
\end{proof}

\begin{lemma}
Let $X$ be a threefold and $\cal F$ a co-rank 1 foliation with non-dicritical terminal 
foliation singularities.
Let
$H$ be a general hyperplane.  Then $(H, \cal F_H)$ has 
terminal foliation singularities.
\end{lemma}
\begin{proof}
The proof of the corresponding statement for varieties works equally well in this case.
\end{proof}

\begin{corollary}
Let $X$ be a threefold and $\cal F$ a terminal Gorenstein co-rank 1 foliation.  Then
$(X, \cal F)$ has at most isolated singularities.
\end{corollary}
\begin{proof}
Follows from the above lemma and the fact that Gorenstein terminal surface foliations
are smooth foliations on smooth surfaces.
\end{proof}

\begin{lemma}
\label{intersectioncomparison}
Let $X$ be a threefold and $\cal F$ a co-rank 1 foliation.
Suppose that $(\cal F, \Delta)$ has canonical
singularities and $(\cal F, \Delta)$ is terminal along $\text{sing}(X)$.
Let $D, D_1, ..., D_n$ be a collection of $\cal F$-invariant
divisors.  Suppose that $D, D_1, ..., D_n$ are $\bb Q$-Cartier.
Let $D^\nu \rightarrow D$ be the normalization.

Write $K_{\cal F}\vert_{D^\nu} = K_{D^\nu}+\Theta$
and $(K_X+D+\sum D_i)\vert_{D^\nu} = K_{D^\nu}+\Delta$.
Then $\Theta \geq \Delta \geq 0$ with 
equality along those centres contained in $\text{sing}(X)$.

In particular, 
if $C\subset D$ is not contained in $\text{sing}(\cal F)$,
then $K_{\cal F}\cdot C \geq (K_X+D+\sum D_i)\cdot C$.
\end{lemma}
\begin{proof}
First, observe that since $\cal F$ is terminal along $\text{sing}(X)$ we see
that $\text{sing}(\cal F)\cap \text{sing}(X)$ does not contain any curve.

Write $\Theta = \sum a_iT_i +\sum b_jS_j$ and $\Delta = \sum c_iT_i+\sum d_jS_j$
where $T_i \subset \text{sing}(\cal F)$
and $S_j \subset \text{sing}(X)$.

Notice that since $\cal F$ has canonical singularities,
$D\cup D_1\cup...\cup D_n$
is normal crossings in codimension 2, \cite[Corollary 3.6]{LPT11}.
Furthermore, since $\cal F$ is terminal along $\text{sing}(X)$
any 2 invariant divisors cannot intersect along $\text{sing}(X)$.

This gives us $D_i\vert_D$ is a reduced divisor 
and that $c_k =1$ for all $k$ such that $T_k \subset \text{supp}(\Delta)$.
$D \cap D_i \subset \text{sing}(\cal F)$ for all $i$
and so $\text{supp}(\sum c_iT_i) \subset \text{supp}(\sum a_iT_i)$.

We have that $a_i \geq 1$, with equality
if and only if either
\begin{enumerate}
\item $\cal F$ is simple at the generic point of $T_i$ and $D$ is a strong separatrix along $T_i$ or
\item $\cal F$ is canonical but not simple at the generic point of $T_i$.
\end{enumerate} 
To see this we may cut by a generic hyperplane, and so we may assume
that $D$ is a curve and $T_i$ is a point. If $T_i$ is a simple singularity
then the claim follows from Example \ref{explicitcomputation}.  So suppose that $T_i$
is canonical but not simple and let $\omega$ be a 1-form
defining the foliation in a neighborhood of $T_i$.  
In this case by the classification
of canonical surface foliation singularities, \cite[Fact I.2.4]{McQuillan08},
we see that $\omega = pxdy-(qy+\epsilon x)dx+h.o.t.$ where $\epsilon \in \{0, 1\}$ and $D = \{x = 0\}$ or $\{y = 0\}$.
In either case $\omega\vert_D$ vanishes to order 1 at $T_i$ and so $a_i =1$.
In particular, from this it follows that $a_i \geq c_i$ for all $i$.

To see that $d_j= b_j$ for all $j$ we may cut by a generic hyperplane, and so we may assume
that $X$ is a surface and that $D$ is a curve.  The equality then follows from
\cite[III.2.bis.1]{McQuillan08}, for example.

Thus, for $C \neq T_i$ for all $i$, we have
that $0> (K_{D^\nu}+\Theta)\cdot C \geq (K_{D^\nu}+\Delta)\cdot C$.
\end{proof}

The next lemma guarantees that we only contract curves tangent to the foliation in the
course of the MMP, which in turn implies that the singularities of our foliation
stay non-dicritical.

\begin{lemma}
\label{onlycurvestangent}
Let $X$ be $\bb Q$-factorial and klt threefold and $\cal F$ 
be non-dicritical, co-rank 1 foliation, 
and suppose furthermore that 
$(\cal F, \Delta)$ is log canonical.

Let $R$ be a $K_{\cal F}+\Delta$-negative extremal ray.  Suppose
that $[C] \in R$.  Then $C$ is tangent to the foliation.
\end{lemma}
\begin{proof}
Suppose for sake of contradiction that there exists a $C \subset X$
with $[C] \in R$ and $C$ is transverse to $\cal F$.  Let $H_R$ be a supporting
hyperplane to $R$.


Suppose first that $\nu(H_R) <\text{dim}(X)=3$. 
As in the proof of the cone theorem, or in the proof of Lemma \ref{l_closedness}
we see through a general point of $X$ there is a rational curve tangent to $\cal F$
spanning $R$.

Let $B$ be the closed subscheme of the Hilbert scheme of $X$ parametrizing those subvarieties
tangent to $\cal F$.  We may consider the following diagram 
as in \cite[\S 4]{Kollar91c}
\begin{center}
\begin{tikzcd}
U \arrow{r}{F} \arrow{d}{p} & X\\
Z &
\end{tikzcd}
\end{center}
where $Z$ is the closed subscheme of $B$ whose general element
paramatrizes a rational curve tangent to $\cal F$ and spanning $R$.
Let 
$p: U \rightarrow Z$ be the universal family over $Z$ and $F:U \rightarrow X$ the evaluation map and observe
that $F$ is surjective. 
After taking general hyperplane sections
of $Z$ we may assume $\text{dim}(Z) = 2$. 

For $x \in Z$ let $V_x = \{y \in Z : F(p^{-1}(x)) \cap F(p^{-1}(y)) \neq \emptyset\}$
and observe that $V_x \subset Z$ is a closed subset. Moreover, by non-dicriticality of $\cal F$
we see that $V_x$ is a strict subset of $Z$ for all $x \in Z$. Indeed, if $\text{dim}(V_x) = 2$ then $F(p^{-1}(x))$ would be a curve
tangent to $\cal F$
meeting infinitely many leaves, a contradiction.

There are therefore two cases for general $x \in Z$
\begin{enumerate}
\item \label{i_0} $\text{dim}(V_x) = 0$; or
\item \label{i_1} $\text{dim}(V_x) = 1$.
\end{enumerate}

Suppose that we are in case \ref{i_1}.
For $x \in Z$ a general point define $D_x = F_*p^*(V_x)$.  Since $\cal F$
is non-dicritical we see that for all $y \in V_x$ that $F(p^{-1}(y))$ all must belong to the same
leaf and thus $D_x$ is in fact a $\cal F$-invariant divisor.
Thus $\cal F$ admits a meromorphic first integral $\phi:X \dashrightarrow S$
where $\text{dim}(S) = 1$. 

We claim that $\phi$ is in fact a morphism.  
Let $\mu:\widetilde{X} \rightarrow X$
be a resolution of singularities of $\phi$, 
let $\widetilde{\cal F}$ be the transformed foliation on
$\widetilde{X}$, let
$\tilde{\phi}:\widetilde{X} \rightarrow S$ be the resolved map
and observe that $\tilde{\phi}$ 
is a first integral for $\widetilde{\cal F}$.  
Since $\cal F$ is non-dicritical then for all $x \in X$
we have that $\mu^{-1}(x)$ is tangent to $\widetilde{\cal F}$
and so we see that $\tilde{\phi}(\mu^{-1}(x))$ must be a single 
point $s \in S$.  The rigidity lemma, \cite[Lemma 4.1.13]{BS95}, then implies
that $\tilde{\phi}$ descends to a morphism $X \rightarrow S$, namely
$\phi$ itself.

By definition we have that if $y \in Z-V_x$ is general then $D_x \cap F(p^{-1}(y)) = \emptyset$ hence $0 = D_x\cdot R = D_x\cdot C$.
However since $C$ is transverse to the foliation and $D_x$ is invariant, for a general $x$
we must have that $C\cdot D_x >0$, a contradiction.

So we are in case \ref{i_0}.  We define $H' = F_*p^* H$ where $H$ is a general ample divisor on $Z$. Observe
that if $y \in Z$ is a general point that $H' \cap F(p^{-1}(y)) = \emptyset$ hence $0 = H'\cdot R = H'\cdot C$.
Notice that $F^{-1}(C) \subset U$ must be a curve, otherwise, through a general point $P \in C$
there is a 1-dimensional family of curves tangent to $\cal F$ passing through $P$, hence the leaf through $P$
is algebraic.  Arguing as in case \ref{i_1}, we see that $\cal F$ would admit
a first integral, a contradiction.
Thus, we see that $p^*H\cdot F^{-1}(C) >0$ and so $H' \cdot C>0$, which is a contradiction.

\medskip

So we have $\nu(H_R) = 3$ and so $H_R \sim_{\bb Q} A_0+E_0$ where $A_0$ is ample and $E_0 \geq 0$.  Since $H_R\cdot R = 0$
and $A_0\cdot R >0$ we see
that $E_0\cdot R<0$ and so there is some irreducible component $E$ of $E_0$ with $E \cdot R <0$.
Thus $E\cdot C <0$ and so $C \subset E$.
Since $C$ is transverse to $\cal F$ we see that $E$ is transverse to $\cal F$.
We may then find $0\leq s_0 \leq 1$ such that $\Delta+s_0E = \Delta'+E$ where $E$ is not contained in the support
of $\Delta'$.

Moreover, we know that $(\cal F, \Delta'+E = \Delta+s_0E)$ is log terminal at the generic point of $C$.
Indeed, suppose otherwise and let $0 \leq t \leq s_0$ be the log canonical threshold of $(\cal F, \Delta)$ with respect
to $E$ at the generic point of $C$.  On one hand, since $E\cdot C < 0$
and $(K_{\cal F}+\Delta)\cdot C<0$ we have that $(K_{\cal F}+\Delta+tE)\cdot C <0$.  On the other hand foliation
subadjunction, Theorem \ref{foliatedsubadjunction}, implies that $(K_{\cal F}+\Delta+tE)\cdot C \geq 0$, a contradiction.

Writing $\nu:E^\nu \rightarrow E$ for the normalization map
and $\cal G$ for the foliation induced on $E^\nu$ we have by Proposition \ref{foliationadjunction},
$K_{\cal G}+\Theta = \nu^*(K_{\cal F}+\Delta'+E)$ where $\Theta \geq 0$.
A computation exactly as in \cite[Propostion 5.46]{KM98} shows that $(\cal G, \Theta)$
is log terminal at the generic point of $C'$ where $C'$ is the strict transform of $C$.
Note that we also have
\[(K_{\cal G}+\Theta)\cdot C'<0.\]


Letting $K_{\cal F}+\Delta'+E+A = H_R$ be a supporting hyperplane for $R$, with
$A$ ample, we have that $K_{\cal G}+\Theta+\nu^*A$ is a nef divisor,
and $(K_{\cal G}+\Theta+(1-\epsilon)\nu^*A)\cdot C'<0$ for $\epsilon >0$.

We claim that $K_{\cal G}+\Theta+\nu^*A$ is not big.  Indeed, suppose for sake of contradiction
that $K_{\cal G}+\Theta+\nu^*A$ were big.  Then for $\epsilon>0$ sufficiently
small we have that $K_{\cal G}+\Theta+(1-\epsilon)\nu^*A$ is big
and $(K_{\cal G}+\Theta+(1-\epsilon)\nu^*A)\cdot C' <0$ and so we may write
$(K_{\cal G}+\Theta+(1-\epsilon)\nu^*A)\sim_{\bb Q} A'+G+aC'$ where $A'$ is ample, 
$G$ is an effective divisor whose support does not contain $C'$ and $a>0$. 

Since $(\cal G, \Theta)$ is log terminal at the generic point of $C'$
we may find $t_0>0$ so that $\Theta+t_0(A'+G+aC') = \Theta'+C'$ where the support
of $\Theta'$ does not contain $C'$.
On one hand we see that $(K_{\cal G}+\Theta'+C')\cdot C'<0$, on the other
hand, foliation adjunction, Proposition \ref{foliationadjunction}, implies that $(K_{\cal G}+\Theta'+C')\cdot C' \geq 0$, a contradiction.

Thus $(K_{\cal G}+\Theta+\nu^*A)^2 =0$ and so we may apply 
foliated bend and break, Corollary \ref{bendandbreak}, as earlier in the proof to produce rational curves $\Sigma$
through a general point of $E^\nu$ tangent
to the foliation with $\nu^*H_R\cdot \Sigma = 0$ and hence the pushforward of these curves span the ray $R$.

By non-dicriticality of $\cal F$, we see that $\cal G$ is the foliation
induced by a fibration in rational curves $E^\nu \rightarrow B$.
Let $f$ be a general fibre of this morphism.  We have that $[f] \in R$.
Thus $\nu^*H_R\cdot f = \nu^*H_R\cdot C' = 0$ and so we may apply Lemma \ref{l_nefdivtriv} below to conclude
that $\nu^*H_R$ is numerically trivial.
However, this implies that $-(K_{\cal G}+\Theta)$ is numerically equivalent to $\nu^*A$.  Since $\nu^*A$
is ample this implies that $-(K_{\cal G}+\Theta)$ is ample as well.
Observe moreover, that $\nu^*H_R$ being numerically trivial implies that if $\Sigma \subset E^\nu$ is a curve
then $[\nu_*\Sigma] \in R$. 

\medskip

The foliated cone theorem
for surfaces, Theorem \ref{conetheoremsurfaces}, applies to $(\cal G, \Theta)$
to show that every extremal ray in $\overline{NE}(E^\nu)$ is spanned by a curve
tangent to the foliation or contained in the non-lc locus of $(\cal G, \Theta)$.

Let $\Sigma$ be some curve contained in the non-lc locus of $(\cal G, \Theta)$.
Then $\nu_*\Sigma$ is a non-lc centre of $(\cal F, \Delta'+E = \Delta+s_0E)$. 
By our above observations we know that $[\nu_*\Sigma] \in R$, and so 
\begin{enumerate}
\item \label{i_e} $E\cdot \nu_*\Sigma < 0$ and
\item \label{i_k} $(K_{\cal F}+\Delta)\cdot \nu_*\Sigma<0$.
\end{enumerate}
Suppose for a contradiction that $\Sigma$ is transverse to $\cal G$.
Let $0 \leq t \leq s_0$
be the log canonical threshold of $(\cal F, \Delta)$ with respect to $E$ at the generic point of $\nu_*\Sigma$. We see
by inequalities \ref{i_e} and \ref{i_k} above that
$(K_{\cal F}+\Delta+tE)\cdot \nu_*\Sigma<0$, but on the other hand foliation
subadjunction, Theorem \ref{foliatedsubadjunction}, 
tells us that $(K_{\cal F}+\Delta+tE)\cdot \nu_*\Sigma \geq 0$, a contradiction.
Thus $\Sigma$ is tangent to $\cal G$.

This implies that $\overline{NE}(E^\nu)$ is spanned by
fibres of $E^\nu \rightarrow B$, an impossibility.
\end{proof}

\begin{lemma}
\label{l_nefdivtriv} 
Let $p:S \rightarrow B$ be a surjective morphism where $S$ is a projective surface and $B$ is a curve.
Let $F \subset S$ be a general fibre of $p:S \rightarrow B$ and let $C \subset S$ be a curve 
which dominates $B$.  Suppose that $H$ is a nef divisor on $S$ with $H\cdot F = H\cdot C = 0$.
Then $H$ is numerically trivial.
\end{lemma}
\begin{proof}
Let $\mu:S' \rightarrow S$ be a resolution of singularities of $S$ and let
$p':S' \rightarrow B$ be the induced fibration. Let $C'$ denote the strict transform of $C$
and $F'$ denote the strict transform of $F$.  Observe that $C'$ still dominates $B$ and $F'$ is still a general
fibre of $p'$.  Moreover, $\mu^*H$ is numerically trivial if and only if $H$ is so we may freely replace
$S$ by $S'$ and so may assume that $S$ is smooth.

For $k\gg 0$ observe that the divisor $D_k = C+kF$ is a big divisor and so we may write
$D_k \sim_{\bb Q} A+E$ where $A$ is ample and $E$ is effective.  Since $H$ is nef and $H\cdot D_k = 0$
this implies $H\cdot E =H\cdot A = 0$.  In particular, it has zero intersection with an ample divisor and is therefore numerically trivial.
\end{proof}

\subsection{Fibre type contractions}
Throughout this section we assume 
$X$ is a projective $\bb Q$-factorial and klt threefold.
Let $\cal F$ be a co-rank 1 foliation with non-dicritical singularities, $\Delta \geq 0$
and let $R$ be a $K_{\cal F}+\Delta$-negative
extremal ray with $\text{loc}(R) = X$. Notice that in this case $R$ is in fact $K_{\cal F}$-negative.

Suppose that $\text{loc}(R) = X$.  Let $H_R$ be a supporting hyperplane of $R$.
In the proof of the cone theorem we see that in this case $\nu(H_R)<3$,
and that there is a general complete intersection curve $C$ such that $K_{\cal F}\cdot C<0$.

\begin{lemma}
Let $\cal E = \cal F\vert_C$.  Either

(i) $\cal E$ is semi-stable and $\cal F$ is a fibration over a curve, or

(ii) $\cal E$ is unstable, and either there is a foliation by rational curves tangent to $\cal F$
realizing $\cal F$ as the pullback of a surface foliation or $\cal F$ is a fibration over a curve.
\end{lemma}
\begin{proof}
If $\cal E$ is semi-stable, then since $\text{det}(\cal E)$ is ample, $\cal E$
is ample as a vector bundle.  \cite{BMc01} applies in this case.

Otherwise, let $\cal L \subset \cal E$ be a maximal destabilizing subbundle.  By the Mehta-Ramanathan
theorem
this extends to a subsheaf $\cal G \subset \cal F$ which is a foliation by rational curves.

Consider the following diagram as in Lemma \ref{onlycurvestangent}
\begin{center}
\begin{tikzcd}
U \arrow{r}{F} \arrow{d}{p} & X\\
Z &
\end{tikzcd}
\end{center}

where $Z$ is projective and
where $p: U \rightarrow Z$ is a projective morphism whose general fibre is a rational
curve tangent to $\cal G$.

For $z \in Z$ let $C_z \coloneqq p^{-1}(z)$ and $D_z\coloneqq F(C_z)$
so $D_z$ is tangent to $\cal G$. 
Let $D_{z_0}$ be a general curve,
then either:

\begin{enumerate}
\item \label{i_emp} $D_{z_0}$ does not intersect any other $D_z$; or

\item \label{i_inf} $D_{z_0}$ intersects infinitely many other $D_z$.
\end{enumerate}

In case \ref{i_emp}
$\cal G$ in fact defines a fibration over a surface.  Since this
fibration is tangent to $\cal F$, we see that $\cal F$ is in fact pulled back from this surface.

In case \ref{i_inf}, if $D_{z}$ meets $D_{z_0}$, then since $\cal F$ is non-dicritical,
$D_z$ and $D_{z_0}$ must (generically) belong to the same leaf.  Since there are infinitely
many $D_z$ meeting $D_{z_0}$, the leaf containing $D_{z_0}$ must in fact be algebraic.  Moreover,
we see that the leaf is covered by rational curves.  In any case, the result follows.
\end{proof}

\begin{theorem}
\label{fibretypecontraction}
Set up as above.
The contraction associated to $R$
exists.
\end{theorem}
\begin{proof}
By our previous lemma we see that $R$ is in fact $K_X$-negative, indeed, if $C$
is a general curve which spans $R$ then $K_{\cal F}\cdot C = K_{X/Z}\cdot C = K_X\cdot C$
where $\pi:X \rightarrow Z$ is the fibration guaranteed by the previous lemma.
The existence of the contraction follows the corresponding statement about
$K_X$-negative contractions.
\end{proof}

\subsection{Divisorial contractions}
Throughout this section we assume $X$ is a projective 
$\bb Q$-factorial and klt threefold.

The idea behind constructing a divisorial contraction of a divisor $D$
is to realize it as a $K_X+(1-\epsilon(D))D$-contraction.

\begin{lemma}
\label{singcomparison}
Let $X$ be as above and suppose that $\cal F$ is a co-rank 1 foliation with 
non-dicritical singularities
and that $(\cal F, \Delta)$ is log canonical (log terminal).
Let $D = \sum D_i$ be an invariant divisor.  Then $(X, \Delta+D)$
is log canonical (log terminal)
\end{lemma}
\begin{proof}
Let $\pi:(Y, \Delta'+D') \rightarrow (X, \Delta+D)$ be
a log resolution of $(X, \Delta+D)$ where $\Delta'$ is the strict transform
of $\Delta$ and $D'$ is the strict transform of $D$.  Perhaps passing to a higher resolution
we may assume that $\cal G$ has simple (hence canonical) singularities where $\cal G$
is the pulled back foliation.

Write
$$K_{\cal G}+\Delta'+\sum E_j^1 = \pi^*(K_{\cal F}+\Delta)+
\sum a_i^0E_i^0 +\sum a_j^1E_j^1$$
and 
$$K_Y+\Delta'+D'+\sum E_i^0+\sum E_j^1 = \pi^*(K_X+\Delta+D)+
\sum b_i^0E_i^0 +\sum b_j^1E_j^1$$
where $E^0_i$ are the invariant $\pi$-exceptional divisors
and $E^1_j$ are the non-invariant $\pi$-exceptional divisors.

By assumption $a_i^0, a_i^1 \geq 0 (>0)$ so if we can show that 
$b^\delta_k \geq a^\delta_k$ the result will follow.

By non-dicriticality 
the foliation restricted to each $E_j^1$ is exactly the fibration
structure $E_j^1 \rightarrow \pi(E_j^1)$, and so if $f_j$ is a general fibre
of $\pi\vert_{E_j^1}$ we know 
that $(K_{\cal G}+E_j^1)\cdot f_j = (K_Y+E_j^1)\cdot f_j$
and so $K_{\cal G}\cdot f_j = K_Y\cdot f_j$.
Notice also that $f_j \cdot (D'+\sum E_i^0) = 0$.

This computation and Lemma \ref{intersectioncomparison}
show that 
$$K_{\cal G} - (K_Y+D'+\sum E_i^0)$$
is $\pi$-nef away from finitely many curves and so the negativity lemma, \cite[Lemma 3.38]{KM98},
applies to show that $a^\delta_k-b^\delta_k \leq 0$ for all $\delta, k$
and the result follows.
\end{proof}

\begin{lemma}
\label{transversecontraction}
Let $X$ be a $\bb Q$-factorial and klt threefold and $\cal F$ a non-dicritical co-rank 1
foliation and suppose that $(\cal F, \Delta)$ is canonical.
Let $R$ be a $K_{\cal F}+\Delta$-negative extremal ray with $\text{loc}(R) = D$.

Suppose that $D$ is transverse
to the foliation.  
Then $K_{\cal F}\cdot R = K_X\cdot R$.
\end{lemma}
\begin{proof}
Let $\cal G$ be the foliation restricted to $D^\nu$ where $\nu:D^\nu \rightarrow D$ is 
the normalization.

Choose $t$ so that $\Delta+tD = \Delta'+D$ where $D$ is not contained in $\text{supp}(\Delta')$
and write $\nu^*(K_{\cal F}+\Delta+tD) = K_{\cal G}+\Theta$.  By non-dicriticality we see that 
$\cal G$ comes from a $\bb P^1$-fibration $D^\nu \rightarrow B$.
If $f$ is a general fibre of $D^\nu \rightarrow B$, which spans $R$, then $K_{\cal G}\cdot f = K_{D^\nu}\cdot f = -2$.

Thus, it suffices to show that if $\nu^*(K_X+\Delta+tD) = K_{D^\nu}+\Theta'$ then $\Theta'\cdot f = \Theta\cdot f$.
We see that if $Y \rightarrow X$ is some high enough resolution
then for centres transverse to the foliation by Lemma \ref{transdiscrep} 
the discrepancy with respect to $(\cal F, \Delta+tD)$
is exactly the discrepancy with respect to $(X, \Delta+tD)$ and so we see that, by construction of the different, 
Proposition \ref{foliationadjunction}, 
components of $\Theta$ and $\Theta'$ transverse to $\cal G$ must have the same coefficient.

Finally, observe that if $B \subset \text{supp}(\Theta+\Theta')$ is a component tangent to $\cal G$
then $f\cdot B = 0$ and so we may conclude $\Theta\cdot f = \Theta'\cdot f$.
\end{proof}

\begin{lemma}
\label{lem_irred}
Let $X$ be a $\bb Q$-factorial and klt threefold and $\cal F$ a non-dicritical co-rank 1
foliation and suppose that $(\cal F, \Delta)$ is log canonical.
Let $R$ be a $K_{\cal F}+\Delta$-negative extremal ray with $\text{dim}(\text{loc}(R)) = 2$.
Then $\text{loc}(R)$ is an irreducible divisor.
\end{lemma}
\begin{proof}
Suppose for sake of contradiction that $\text{loc}(R)$ contains two distinct
irreducible components $D_1$ and $D_2$.  
We may assume without loss of generality that $D_1$ is a divisor.
By definition for $i =1, 2$
we may find a family of curves $\{C^i_t\}$
covering $D_i$ such that $[C^i_t] \in R$.

Observe that $H_R-\epsilon A \sim_{\bb Q} E \geq 0$ where $H_R$ is the supporting hyperplane
to $R$, $A$ is ample and $\epsilon >0$ is sufficiently small.  Thus there is some irreducible 
component $E_0$
of $E$ with $R\cdot E_0<0$, in particular $C^1_t\cdot E_0<0$
and we see that $E_0 = D_1$ and so $C^1_t \cdot D_1<0$.
On the other hand, we have that $C^2_t\cdot D_1 \geq 0$ but 
this is a contradiction of the fact that $[C^2_t] \in R$.
\end{proof}

\begin{theorem}
Let $X$ be a $\bb Q$-factorial and klt threefold.
Suppose that $\cal F$ is a co-rank 1 foliation
with non-dicritical singularities and $\Delta \geq 0$.  Suppose furthermore that $(\cal F, \Delta)$
has canonical singularities and that $(\cal F, \Delta)$ is terminal along $\text{sing}(X)$.
Let $R$ be a $K_{\cal F}+\Delta$-negative extremal ray and suppose 
that $\text{loc}(R) = D$ a divisor.
Then there exists a contraction of $D$.
\end{theorem}
\begin{proof}
By Lemma \ref{lem_irred} $D$ is irreducible,
thus it is either transverse to the foliation or is invariant by the foliation.

If $D$ is transverse to the foliation, then
by
Lemma \ref{transversecontraction}
we know that $(K_{\cal F}+\Delta)\cdot R = (K_X+\Delta)\cdot R$ and so $R$
is in fact $K_X+\Delta$-negative.  
Otherwise, $D$ is invariant and by Lemma \ref{intersectioncomparison} we see that
$R$ is $K_X+\Delta+D$-negative.

In either case, we see that $R$ is a $K_X+\Delta+(1-\epsilon(D))(1-t)D$-negative extremal
ray for $t>0$ sufficiently small.
By Lemma \ref{singcomparison} we know
that $(X, \Delta+(1-\epsilon(D))(1-t)D)$ is klt and so the contraction of $R$ exists
by \cite[Theorem 3.7]{KM98}.
\end{proof}

\begin{corollary}
Let $c_R:X \rightarrow Y$ be the contraction constructed above.

(i) $\rho(X/Y) = 1$.

(ii) $Y$ is $\bb Q$-factorial and projective.

(iii) $Y$ is klt.
\end{corollary}
\begin{proof}
This is a consequence of standard facts about log contractions, see for instance \cite[Theorem 3.7]{KM98}.
\end{proof}

\subsection{Flipping contractions}
Throughout this section we assume $X$ is a projective $\bb Q$-factorial and klt threefold.
Let $\cal F$ a co-rank 1 foliation and
let $R$ be a $K_{\cal F}+\Delta$-negative 
extremal ray.  Suppose $\text{loc}(R)$ is 1-dimensional and 
let $H_R$ be a supporting hyperplane to $R$.

In this section we show that a flipping contraction can be realized in the
category of algebraic spaces.  While we are able to show that this contraction
(and the flip) can be realized in the category of projective spaces in
some special cases, we are unable to deduce a complete flip theorem.

For a Cartier divisor $D$ let 
$\text{Null}(D) = \{P : P \in V, V \cdot D^{\text{dim}(V)} = 0\}$,
and $BS(D)$ denotes the stable base locus of $D$, i.e. 
$\cap_{m \geq 0} bs(mD)$ where $bs(mD)$ is the base locus of $mD$.
It is easy to see that $BS(D) = bs(mD)$ for $m$ sufficiently large
and divisible.

We will make use of the following result \cite[Corollary 1]{CL14}.
\begin{lemma}
Let $X$ be normal threefold and let $L$ be big and nef.
Let $A$ be an ample divisor.  Then for all $\epsilon >0$ sufficiently small
$\text{Null}(L) = BS(L-\epsilon A)$.
\end{lemma}

  As we have seen
$H_R$ is big and nef.  By our next lemma
$\text{Null}(H_R)$ is a finite collection of curves.

\begin{lemma}
\label{surfaceintersection}
Set up as above.
Let $S \subset X$ be a surface.  Then $H_R^2\cdot S >0$.
\end{lemma}
\begin{proof}
Suppose for sake of contradiction that there is some
surface $S$ such that $H_R^2 \cdot S = 0$.

Let $f: S^\nu \rightarrow S$ be the normalization of $S$.

$H_R\vert_S$ is nef and so it is pseudo-effective, and we proceed case by case
on the numerical dimension of $f^*H_R$.

If $\nu(f^*H_R) = 0$ then $H_R$ is zero on a moving
curve, hence is zero on infinitely many curves, a contradiction.

If $\nu(f^*H_R) = 1$ then write $H_R = K_{\cal F}+\Delta+A$
where $A$ is ample.  We have that $f^*(H_R)^2 =0$,
and that $f^*(H_R)$ has positive intersection
with any ample divisor on $S^\nu$ (otherwise $H_R$ would
be zero on a moving curve).

Thus $$f^*(K_{\cal F}+\Delta)\cdot f^*H_R 
= -f^*A\cdot f^*H_R<0.$$

Perhaps rescaling $H_R$ by a positive 
constant we may write $H_R = A'+D+S$ where
$A'$ is ample, and $D$ is effective, and the support of $D$
does not contain $S$.
Then 
$$f^*H_R\cdot f^*S = -f^*H_R \cdot f^*(A'+D) 
\leq -f^*H_R\cdot f^*A'<0.$$

If $S$ is $\cal F$ invariant, then 
$f^*(K_{\cal F}+\Delta) = K_{S^\nu}+\Theta$ where $\Theta \geq 0$.

We apply bend and break, Lemma \ref{bendandbreak}, to $D_1 = D_2 = f^*H_R$.
$D_1\cdot D_2 = 0$ by supposition, and by our above computation
$(K_{S^\nu}+\Theta)\cdot D_1 = f^*(K_{\cal F}+\Delta)\cdot f^*H_R <0$.  
Thus, we get through a general
point of $S$ a rational curve $\Sigma$ with $0 = D_2\cdot \Sigma = H_R\cdot \Sigma$
a contradiction.

If $S$ is not $\cal F$ invariant then 
choose $0 \leq t \leq 1$ so that $\Delta+tS = \Delta'+S$ where $S$ is not contained in $\text{supp}(\Delta')$.
By foliation adjunction
we see that $f^*(K_{\cal F}+\Delta+tS) = K_{\cal F_{S^\nu}}+\Theta$, $\Theta \geq 0$.
Again, by our above computations we have that 
$(K_{\cal F_{S^\nu}}+\Theta)\cdot f^*H_R = 
f^*(K_{\cal F}+\Delta+tS)\cdot f^*H_R<0$.

Again we apply bend and break with $D_1 = D_2 = f^*H_R$.
$D_1\cdot D_2 = 0$ by assumption, and $D_1\cdot(K_{\cal F_{S^\nu}}+\Theta)<0$.
Again, through a general point of $S$ we get a rational curve $\Sigma$
tangent to the foliation with $0 = D_2\cdot \Sigma = H_R\cdot \Sigma$
a contradiction.

If $\nu(f^*H_R) = 2$ then $f^*H_R^2 = H_R^2\cdot S >0$ and we are done.
\end{proof}

\begin{lemma}
Set up as above.
$\text{loc}(R)$ can be contracted
in the category of algebraic spaces.
\end{lemma}
\begin{proof}
By our previous lemma $\text{Null}(H_R)$
is a finite collection of curves, each of which span $R$.
Let $A$ be an ample divisor and choose $\epsilon$ sufficiently
small and $m$ sufficiently large so that
$\text{Null}(H_R) = bs(m(H_R - \epsilon A)) = B$.

Let $g:Y \rightarrow X$ be a resolution of the
base locus of $m(H_R - \epsilon A)$
so that we have $g^*(m(H_R - \epsilon A)) = M+F$
where $M$ is semi-ample, $F$ is effective and
$g(F) = B$ and $g(\text{exc}(g)) = B$.

Let $G$ be an effective divisor $\bb Q$-Cartier divisor supported on $\text{exc}(g)$
such that $-G$ is $g$-ample (such a $G$ exists because $X$ is $\bb Q$-factorial).

Then for $1 \gg \delta>0$ we have that $\epsilon g^*A - \delta G+M = A'$ is ample on $Y$
and $(F+\delta G) +A' = mg^*H_R$.  Since $g^*H_R\vert_{\text{exc}(g)}$
is trivial,
we see that $(F+\delta G)\vert_{\text{exc}(g)} = -A'\vert_{\text{exc}(g)}$ is anti-ample.

Let $N$ be such that $N(F+\delta G) = D$ is an integral Cartier divisor.
By construction we know that $\cal O_{D_{red}}(-D)$ is ample. 
Recall that a line bundle $L$ on a scheme $Z$ is ample if and only
if $L\vert_{Z_{red}}$ is ample where $Z_{red}$ is the reduction, see for example
\cite[III.ex 5.7.(b)]{Hartshorne77}.

Thus, $\cal O_D(-D)$ is ample and so
 $D$ is a subscheme with an anti-ample normal bundle,
and so it may be contracted to a point by \cite[Theorem 6.2]{Artin70}.

This contraction
factors through $g$ and
gives a contraction
$X \rightarrow Z$. By \cite{Artin70} this contraction may
be taken in the category of algebraic spaces.
\end{proof}

Note that we have not proven that $Z$ is projective.
However, we do have the following criterion which will be useful later.

\begin{lemma}
Set up as above.
Let $H_R$ be the supporting hyperplane
to $R$.  Assume that $H_R$ descends to a $\bb Q$-Cartier 
divisor on $Z$, then $Z$ is projective.
\end{lemma}
\begin{proof}
By assumption, if $f$ is the contraction, then $H_R = f^*M$.
We claim that $M$ is in fact ample.  First $M$ is nef
and $M^3 >0$.  If $C$ is any curve in $Z$ then we also
have $M\cdot C >0$.  By Lemma \ref{surfaceintersection} if $S$ is any surface
then $M^2\cdot S >0$.  Thus the Nakai-Moishezon criterion
for ampleness applies to show that $M$ is ample, and so $Z$
is projective.
\end{proof}

\begin{remark}
Notice that we do not know if the flipping contraction is a $K_X+D$-contraction for some 
divisor $D$.
Indeed, the natural choice for $D$, namely, the germ of a leaf around a flipping curve
might not be a $\bb Q$-Cartier, even if $X$ is (algebraically) $\bb Q$-factorial,
and so a $K_X+D$-contraction doesn't even make sense.
We will revisit this issue when constructing terminal flips.
\end{remark}

\section{The smooth MMP and the existence of terminal flips}

In this section we first give a classification of flipping curves
when $X$ is a smooth threefold, 
and then use this to deduce the existence of an MMP starting with a smooth
foliation on a smooth threefold.

Finally, we explain how to construct flips for terminal co-rank 1 foliations on threefolds.

\subsection{Flipping curves on smooth varieties}
We begin with an example showing that this case can really happen, 
see also \cite{BruP11}
for some similar examples:

\begin{example}
Let $\phi: X_1 \dashrightarrow X_2$ be the threefold toric flop.
We can realize $X_i$ as an $\bb A^1$-bundle over $\bb A^2$ blown up
at a point, with exceptional curve $C_i$.
Let $\cal G$ be a foliation on $\bb A^2$ blown up at a point
so that the exceptional curve is invariant, meets exactly two
other invariant curves and $\cal G$ has canonical singularities.

Let $\cal F_2$ be the pull back of this foliation to $X_2$.
Let $\cal F_1$ be the strict transform of $\cal F_2$ under
$\phi^{-1}$.  $C_1$ is the flop of $C_2$, and observe that
$C_1 \subset \text{sing}(\cal F_1)$, and $\cal F_1$
has canonical singularities along $C_1$.

Let $X_0$ be the blow up of $X_i$ along $C_i$,
with exceptional divisor $E$.
and $\cal F_0$ the transformed foliation on $X_0$.
Let $\pi_i: X_0 \rightarrow X_i$.
Let $\tilde{C}_1$ be a $\bb P^1$ sitting above $C_1$.

Observe that $\cal F_0\vert_E = K_E + \Delta$
where $\Delta$ consists of three of the four torus
invariant divisors on $E$.
Thus, $K_{\cal F_0} \cdot \tilde{C}_1 = -1$
and since $\pi^*K_{\cal F_1} = K_{\cal F_0}$
we get that $K_{\cal F_1}\cdot C_1 = -1$.

Furthermore, we can check that $K_{\cal F_2}\cdot C_2 = 1$.

Thus, we see that $C_1$ is an isolated $K_{\cal F_1}$-negative
extremal ray, and the flip of $C_1$ exists.
\end{example}

We will use the following local version of Reeb stability:

\begin{lemma}
Let $L$ be a leaf of a foliation $\cal F$ on $X$
 and $K \subset L$ a compact
subset.  Suppose that $K$ is simply connected.  Then 
there is an open subset of $X$,
$K \subset W \subset X$ and a holomorphic submersion $W \rightarrow U$
such that the leaves of $\cal F$ are given by the fibres of this map.
\end{lemma}
\begin{proof}
The usual proof of Reeb stability, see for example \cite[Theorem 2.9]{MM},
works in this case.
\end{proof}

\begin{corollary}
\label{smoothflippingcurves}
Let $X$ be a smooth threefold and $\cal F$ be a co-rank foliation on $X$ and 
suppose that $\cal F$ has simple singularities.
Let $C$ span an $K_{\cal F}$-negative extremal ray of flipping type.
Then $C$ is contained in $\text{sing}(\cal F)$.
\end{corollary}
\begin{proof}
Suppose the contrary.
We first claim that $C$ is actually disjoint
from the singular locus.  Let $S$ be a germ of a leaf containing
$C$, 
then since $C$ is a flipping curve we must have $C^2<0$ in $S$.
Adjunction tell us that $(K_S+C)\cdot C = 2g(C)-2<0$
and so $C$ is a smooth rational curve and $K_S\cdot C = C^2 = -1$.

On the other hand, by Lemma \ref{l_summarylemma} we know that $K_{\cal F}\vert_S = K_S+\Delta$ where $\Delta$
is a divisor with integer coefficients whose support is exactly $\text{sing}(\cal F)\cap S$.
Thus if $K_{\cal F}\cdot C <0$ we must have $\Delta\cdot C =0$ and so $C \cap \text{sing}(\cal F) = \emptyset$.

Now apply local Reeb stability as above to see that $C$ actually
moves to near by leaves, giving our contradiction.

J. V. Pereira has given the following alternative proof:
Suppose as above that $C$ is not contained in the singular locus.
By restricting Bott's partial connection on the leaf to $C$
and noting that $N_{C/S} = \cal O(-1)$ we see that
$N_{C/X} = \cal O(-1) \oplus \cal O$.
Thus since $C$ is a smooth rational curve $H^0(C, N_{C/X}) \cong \bb C$ and $H^1(C, N_{C/X}) = 0$.
This latter equality implies that the first order deformations of $C$
are unobstructed and so give a deformation of $C$ in $X$, i.e. $C$ moves in $X$,
a contradiction.
\end{proof}

\begin{corollary}
\label{notuniruledloc}
Let $X$ be a smooth threefold and $\cal F$ be a smooth
co-rank 1 foliation on $X$.  Let $R$ is a $K_{\cal F}$-negative
extremal ray and let $H_R$ be a supporting hyperplane to $R$.  Suppose $\nu(H_R) = 3$.  
Then $\text{loc}(R)$ is a divisor transverse to the foliation.

Moreover $\text{loc}(R)$ is a $\bb P^1$-bundle over a curve $B$
and there is a morphism $\mu:X \rightarrow Y$ contracting $\text{loc}(R)$ to $B$
and where $Y$ is smooth.
\end{corollary}
\begin{proof}
By Corollary \ref{smoothflippingcurves} we know that $\text{loc}(R) = D$ must be divisor.
Suppose for sake of contradiction that $D$ is foliation invariant,
then $D$ is covered by rational curves which by Reeb stability can be moved
into nearby leaves and therefore $\cal F$ is uniruled.

Thus $D$ is transverse to the foliation
and $D$ is covered by curves tangent to the foliation and so if $\cal G$
is the foliation restricted to $D$ we see that $\cal G$ gives a $\bb P^1$-fibration
structure $D \rightarrow C$ on $D$ which induces $\cal G$ where a general
fibre spans $R$. In particular, we see that $K_{\cal G}\cdot f = K_D\cdot f<0$ where
$f$ is a general fibre of $D \rightarrow C$.

As in the proof of Lemma \ref{transdiscrep} we see that this implies 
that $R$ is in fact $K_X$-negative, and so by \cite[Theorem 1.1]{Kollar91c} 
we see that the fibration $D \rightarrow C$ is in fact a $\bb P^1$-bundle
and that moreover we may contract $D$ by a morphism to a smooth variety.
\end{proof}

\subsection{Running the smooth MMP}

\begin{theorem}
\label{smoothMMP}
Let $X$ be a smooth threefold and $\cal F$ a smooth co-rank 1 foliation.
Then there is a foliated MMP for $(X, \cal F)$.
\end{theorem}
\begin{proof}
Let $R$ be any $K_{\cal F}$-negative extremal ray.
If $\text{loc}(R) = X$ then the contraction exists by Theorem \ref{fibretypecontraction}
and we stop the MMP.  

Otherwise,
by Corollary \ref{notuniruledloc}
we have that $\text{loc}(R) = D$ is a
divisor transverse to the foliation and 
we may contract $D$ to a smooth curve
$C$ by a morphism $\pi:(X, \cal F) \rightarrow (Y, \cal G)$ and where $\cal G$
is the transformed foliation.
We claim
that $\cal G$ is smooth.

Away from $C$ this immediate, and since $D$ is transverse to the foliation
we must have that $\cal G$ is smooth at the generic point of $C$ and so
$\cal G$ has at worst isolated singularities along $C$.  
However, since $\pi$ is just the blowing up in a smooth curve transverse to the foliation
a direct computation shows that
if $Q \in C$ is a singularity of $\cal G$ then $\pi^{-1}(Q)$ is a singularity
of $\cal F$, an impossibility and so the claim is proven.

Thus, we can perform the contraction in the category of smooth foliations
on smooth varieties, allowing us to proceed with the MMP.  Since each contraction
drops the Picard number by 1, this process must eventually terminate.
\end{proof}




\begin{corollary}
Let $X$ be a smooth threefold and let $\cal F$ be a smooth rank 2 foliation
on $X$ with $\nu(K_{\cal F}) = 0$.  Then $(X, \cal F)$ is birational to $(Y, \cal G)$
where $(Y, \cal G)$ is one of the following:

\begin{enumerate}
\item Up to a finite cover $\cal G$ is the product of a fibration in Calabi-Yau manifolds
and a linear foliation on a torus.

\item $\cal G$ is transverse to a $\bb P^1$-bundle

\item Up to a finite cover $\cal G$ is induced by a fibration $X' \times C \rightarrow C$
where $c_1(X') = 0$.

\end{enumerate}
\end{corollary}
\begin{proof}
Run an MMP $f:(X, \cal F) \rightarrow (Y, \cal G)$.  We have that $c_1(\cal G) = 0$,
and so we may apply \cite[Theorem 1.2]{Touzet08}.
\end{proof}

\subsection{Existence of terminal flips}

We will need to make use of the following generalization of Malgrange's 
theorem due to Cerveau and Lins-Neto \cite[Corollary 1]{CLN08}. 

\begin{lemma}
\label{cln08}
Let $X$ be a germ of an analytic variety at $0 \in \bb C^N$
of dimension $n \geq 3$, and let $\cal F$ be a holomorphic
foliation on $X^* = X-\text{sing}(X)$.  Suppose
that:

1) $X$ is a complete intersection, 

2) $\text{dim}(\text{sing}(X)) \leq n-3$, 

3) $\cal F$ is defined by a holomorphic 1 form $\omega$ on $X^*$
such that $\text{dim}(\text{sing}(\omega)) \leq n-3$.

Then $\cal F$ has a holomorphic first integral. 
\end{lemma}

The following proposition can also be viewed as being a version
of Malgrange's theorem:

\begin{lemma}
\label{novelmalgrange}
Let $0 \in X$ be a threefold germ with a co-rank 1 foliation $\cal F$.
Suppose $X$ has log terminal singularities 
and that $\cal F$ has terminal singularities.  Then $\cal F$ has a holomorphic
first integral.
\end{lemma}
\begin{proof}
Take a Galois quasi-\'etale cover with Galois group $G$ 
(ramified only over $\text{sing}(X)$) so that
$K_X, K_{\cal F}$ are both Cartier, \cite[Theorem 1.10]{GKP16}.  Call
this cover $\pi: (Y, \cal G)\rightarrow (X, \cal F)$.  
$K_Y = \pi^*K_X$ and $K_{\cal G} = \pi^*K_{\cal F}$,
since $Y$ is log terminal and $K_Y$ is Cartier, this implies that $Y$ is canonical.
Since $\pi$ is \'etale in codimension 1, we see that $\cal G$ is terminal.

Next we claim that $Y$ is actually terminal.  Notice
that $(N^*_{\cal G})^{**}$ is a line bundle being the difference of 
2 Cartier divisors,
and since $Y$ is log terminal, we have that the foliation
discrepancies are less than or equal to the usual discrepancies, Lemma
\ref{smoothdiscrep}. 
Thus, since $\cal G$ is terminal, this immediately implies
that $Y$ is terminal.

$Y$ is terminal and index 1, which implies by \cite[5.38]{KM98}
that it is a cDV hypersurface singularity, in particular $Y$
is a complete intersection and $\text{sing}(Y)$ is isolated.

Notice also that $\cal G$ is smooth
away from $\text{sing}(Y)$.  We claim that $\cal G$
has a holomorphic first integral.

Observe that for any $0 \in \text{sing}(Y)$, if we write $Y^* = Y - 0$
that
$\cal G$ is defined by global a 1-form on $Y^*$ near $0$.
Indeed, take any generator $\omega$
of $(N^*_{\cal G})^{**}$ around $0$.  
Observe
that since $\cal G$ is smooth away from $0$ we have that
$\text{sing}(\omega) \subset \{0\}$.

Thus, \ref{cln08}
applies to show that $\cal G$ has a holomorphic first integral,
i.e. there is a holomorphic fibration $f:Y \rightarrow \bb C$
whose fibres determine $\cal G$.  Notice that $f$ is in fact $G$-equivariant
and so descends to a first integral $X \rightarrow \bb C$.
\end{proof}

\begin{corollary}
Let $X$ be a threefold with a co-rank 1 foliation $\cal F$.
Suppose $(X, \Delta)$ has klt singularities for some divisor $\Delta$, 
and that $\cal F$ has terminal singularities. 
Let $C$ be curve tangent to the foliation
and $S$ a germ of a leaf containing $C$.
Then $S$ is $\bb Q$-Cartier.
\end{corollary}
\begin{proof}
This can be checked analytic locally around points of $C$,
so replace $X$ by a germ around some point $p \in C$.
In this case by Lemma \ref{novelmalgrange}
there exists a holomorphic first integral $f: (p \in X) \rightarrow (0 \in \bb C)$
and we have that $S = f^{-1}(0)$.  The result follows.
\end{proof}

\begin{theorem}
Let $X$ be a $\bb Q$-factorial klt threefold.
Suppose that $\cal F$ is a co-rank 1 foliation with terminal singularities.
Let $f:X \rightarrow Z$ be a flipping contraction.
Then the flip exists.
\end{theorem}
\begin{proof}
First, observe for all 
points $P \in X$ by Lemma \ref{novelmalgrange} $\cal F$ admits (locally) a holomorphic
first integral. In particular, $\cal F$ has non-dicritical singularities.

We may assume that we are working in the neighborhood
of a connected component of $\text{exc}(f)$, call it $C$.
We have the following sequence
$$0 \rightarrow \text{Pic}(Z)\otimes \bb Q \xrightarrow{f^*} 
\text{Pic}(X) \otimes \bb Q \rightarrow \bb Q$$
where the last arrow is given by intersecting with $C$.
Thus, to prove $\rho(X/Z) = 1$ it suffices to show that if
$M\cdot C= 0$ then $M = f^*M'$ for some $M'$.

By Lemma \ref{onlycurvestangent} we see that $C$ is tangent to $\cal F$ and so
by Lemma \ref{l_summarylemma} we may find $S$, a germ of a separatrix, around $C$.
Let $p = f(C)$.
 and let $V \subset Z$ be a small neighborhood around $p$ so that $S$ is defined on $W = f^{-1}(U)$.
By Lemma \ref{singcomparison} we know that $(W, S - \epsilon S)$ is klt for $\epsilon$ sufficiently small,
and that $-(K_W+S)$ is $f$-ample.
For $n\gg0$ we have that $nM\vert_W - (K_W+(1-\epsilon)S)$ is $f$-ample,
and so the relative analytic base point free theorem, see for example 
\cite[Theorem 3.24]{KM98} or \cite{Nakayama87}, applies
to show that $M\vert_W$ is semi-ample over $W$.
Hence for some sufficiently large $n$, since $M$ is $f$-trivial, $nM$ is pulled back
from a Cartier divisor $nM'$ on $V$.  Since $f$ is an isomorphism 
away from $C$, it is easy to extend $nM'$ to a Cartier divisor on all of
$Z$, and the result follows.  In particular, $Z$ is projective.

Next, run a $K_W+(1-\epsilon)S$-MMP over $V$, and let $f': (W', \cal F') \rightarrow V$
be the output of this MMP, and so $K_{W'}+S'$ is $f'$-nef, where $S'$ is the strict transform
of $S$.

Notice that since $\cal F$ has terminal singularities 
$K_W+S$ is numerically equivalent over $V$ to $K_{\cal F}$, and so $K_{W'}+S'$ is numerically
equivalent to $K_{\cal F'}$ over $V$.  Thus we see
$K_{\cal F'}$ is $f'$-nef.  Furthermore $\cal F'$ is terminal.
In general, $(W', \cal F')$ will not be flip, the flip will be the canonical model of $\cal F'$
over $V$.

We claim that we can construct this canonical model.
We have $S$ is $f$-numerically equivalent to $K_{\cal F} - K_X$. Since $f$ is a contraction
of Picard rank 1, this implies that $S =_{num} \lambda K_{\cal F}$ for
some $\lambda \in \bb Q$, and so $S' =_{num} \lambda K_{\cal F'}$.
In particular for $1 \gg \delta >0$ and $m \gg 0$
we know that $mK_{\cal F'} + \delta S'$ is nef over $V$.

Thus for large enough $m$, and
small enough $\delta$, $mK_{\cal F'} - (K_{W'}+(1-\delta)S')$ is big and nef
over $V$, and $(X', (1-\delta)S')$ is klt.  Thus, we may apply the relative base point free
theorem to conclude that $K_{\cal F'}$ is semi-ample over $Z$.

Let $\phi:(W', \cal F') \rightarrow (W^+, \cal F^+)$ be the corresponding map over $Z$.
Since $W' \rightarrow Z$ is small, we see that $\phi$ is small, and
since $K_{\cal F^+} = \phi_*K_{\cal F}$ we get that $K_{\cal F^+}$
is ample over $V$.

By \cite{Artin70}, we can realize the flip $X \dashrightarrow X^+$
in the category of algebraic spaces,
and since $K_{\cal F^+}$ is ample over $Z$ we get that $X^+$ is in fact projective.
\end{proof}

\section{Toric foliated MMP}

\begin{defn}
Let $X$ be a toric variety.  Let $\cal F$ be a foliation on $X$.
We say that $\cal F$ is toric provided that it is invariant under
the torus action on $X$.
\end{defn}

\begin{lemma}
Let $\cal F$ be a co-rank 1 toric foliation.
Then $K_{\cal F} = - \sum D_\tau$ where the sum is over all
the torus invariant and non-$\cal F$-invariant divisors.
\end{lemma}
\begin{proof}
By passing to a toric resolution 
$\pi: (X', \cal F') \rightarrow (X, \cal F)$, and noting
that the strict transform of a divisor $D$ is torus and $\cal F'$-invariant
if and only if it is torus and $\cal F$-invariant, we see that it suffices
to prove the result on a resolution of $X$.  Thus we may assume
that $X$ is smooth.

Observe that $\cal F$ is defined by a rational torus invariant
1-form $\omega$.  Working in torus coordinates $x_1, ..., x_n$,
we see that $\omega = \sum \lambda_i \frac{dx_i}{x_i}$ where $\lambda_i \in \bb C$.
$\lambda_i \neq 0$
if and only if the divisor associated to $\{x_i = 0\}$ is foliation
invariant. In particular $\omega$ has a pole of order 1 along a torus invariant divisor
if and only if it is a $\cal F$ invariant divisor.
Thus $N_{\cal F}$ is equivalent to the sum of the torus and foliation
invariant divisors, and the result follows.
\end{proof}

\begin{defn}
Let $\sigma$ be a cone in a fan $\Delta$ defining a toric variety.
Let $D(\sigma)$ denote the closed subvariety corresponding to $\sigma$.
\end{defn}

\begin{remark}
We note that if $\tau = \langle v_1, ..., v_n \rangle$
is a full dimensional cone in the fan defining
$X$, then this argument in fact shows $D(v_i)$
is $\cal F$-invariant for some $i$.

Furthermore, if $w = \langle v_1, ..., v_{n-1} \rangle$
is a codimension 1 cone in the fan, then $D(w)$ is tangent
to $\cal F$ if and only if $D(v_i)$ is invariant for some $i$.
\end{remark}


We also make the following simple observation:
\begin{proposition}
\label{basicobservation}
Let $\cal F$ be a co-rank 1 toric foliation on a toric variety $X$.
Suppose that $\cal F$ is defined by 
$\omega = \sum_{i = 1}^n \lambda_i \frac{dx_i}{x_i}$.
If $\lambda_i = 0$ for some $i$, then $\omega$, and hence $\cal F$,
is pulled back along some dominant rational map $f:X \dashrightarrow Y$
with $\text{dim}(Y) <\text{dim}(X)$.
In particular, if $K_{\cal F}$ is not nef, then $\cal F$ is a pull back.
\end{proposition}
\begin{proof}
The first claim is easy.

To prove the second claim,
suppose that $\cal F$ is not a pull back along a dominant rational map.
This remains true after passing to a resolution of singularities of 
$X$.  Let $\cal F'$ be the transformed foliation.  Since $\cal F'$
is not a pull back we have that every torus invariant divisor
is also $\cal F'$ invariant and so $K_{\cal F'}$ is in fact trivial.
Pushing forward gives the result.
\end{proof}

We show that the cone theorem holds for co-rank 1 toric foliations
in all dimensions- first we have the following result:

\begin{theorem}
Let $\cal F$ be co-rank 1 toric foliation with canonical singularities on a toric variety $X$.
Let $C$ be a curve in $X$, and $K_{\cal F}\cdot C<0$, then
$[C] = [M]+\alpha$
where $M$ is a torus invariant curve tangent to
the foliation, and $\alpha$ is a
pseudo-effective class.
\end{theorem}
\begin{proof}
By \cite{Matsuki02}, we can write
$$C = \sum_{\text{tangent to } \cal F} a_uD(u) +
\sum_{\text{not tangent to }\cal F} b_wD(w)$$
and where $u, w$ run over the codimension 1 subcones of the
fan.

We show that some $a_u$ can be taken to be non-zero.
Assume the contrary, that $a_u = 0$ for all $u$.

Since $D(w)$ is not tangent to the foliation, we have that if
$w = \langle v_1, ..., v_{n-1} \rangle$,
then all the $D(v_i)$ are not foliation invariant.

In order to have $K_{\cal F}\cdot C < 0$, we must have
$D(w)\cdot D(v_i)>0$ for some $w, v_i$.  
Let $\tau, \tau'$ be the
two full dimensional cones which are spanned by $w$ and $v_n, v_{n+1}$
respectively.
Then $\tau \cup \tau'$ must be
concave along $\langle v_1, ..., \hat{v_i}, ..., v_{n-1} \rangle$.

Thus, there must be $\sigma_1, ..., \sigma_r$ cones in our fan such that
$\tau \cup \tau' \bigcup_{i = 1}^r \sigma_i$ is a convex subcone of our fan.
Furthermore, we know that both 
$D(v_n)$ and $D(v_{n+1})$ are foliation invariant.
By \cite{Matsuki02}
$\langle v_1, ..., \hat{v_i}, ..., v_{n-1}, v_n \rangle$ or
$\langle v_1, ..., \hat{v_i}, ..., v_{n-1}, v_{n+1} \rangle$
are in the same extremal ray as $D(w)$, and both correspond to torus
invariant curve tangent to the foliation.
\end{proof}

\begin{corollary}
Let $\cal F$ be a co-rank 1, toric foliation with
canonical singularities on a toric variety $X$.
Then, $\overline{NE}(X)_{K_{\cal F}<0} = \sum \bb R^+[M_i]$ where the $M_i$
are torus invariant rational curves tangent to the foliation.
\end{corollary}

\begin{lemma}
Let $\cal F$ be a co-rank 1 toric foliation with canonical non-dicritical singularities
on a toric variety $X$.
Let $R$ be a $K_{\cal F}$-negative extremal ray.
Then there is a contraction corresponding to this extremal ray,
and falls into one of the following types:

\begin{enumerate}
\item fibre type contractions,

\item divisorial contractions, or

\item small contractions.
\end{enumerate}

Furthermore, if a subvariety $Y$ is contracted, it is tangent
to the foliation.
In particular, if $\cal F$
has non-dicritical singularities, the resulting
foliation will still have non-dicritical singularities.
\end{lemma}
\begin{proof}
We know that the contraction exists, what is unclear if the curves
being contracted are tangent to the foliation.
By our cone theorem for toric foliations we know that some curve contracted
is tangent to the foliation, however, it might be the case
that there is a contracted curve transverse to the foliation.

Suppose that $\pi: X \rightarrow Z$
is the contraction and $Y$ is a general fibre.
Suppose for sake of contradiction
that $Y$ is not tangent to the foliation.  Then there is
an induced foliation on $Y$, call it $K_{\cal G}$.
It is standard that $\rho(Y) = 1$, and so $-K_{\cal G}$ is ample.
However, by \ref{basicobservation} we see that since $\cal G$ is birational to a pullback
from a lower dimensional variety,
$\cal G$ must have dicritical singularities, implying that $\cal F$
does as well- a contradiction.  Finally, if a general fibre is tangent
to the foliation, then every fibre is.
\end{proof}

We now handle the flipping case:

\begin{lemma}
Set up as above.
In the case of a small contraction,
the flip exists and no infinite sequence of flips exists.
Furthermore, if $\cal F$ has canonical and non-dicritical
singularities, then the flipped foliation, $\cal F^+$ does as well.
\end{lemma}
\begin{proof}
The existence and termination of the flip can be seen by the fact
that toric log flips exist and terminate.
By the negativity lemma, $\cal F^+$
has canonical singularities.

What remains to be shown is the claim about the non-dicriticality
of $\cal F^+$.  Let $S$ be the flipping locus of $\cal F$.
Let $\pi:X \rightarrow Z$ be the flipping contraction.
As noted, $\pi$ only contracts curves tangent to the foliation.

Suppose for sake of contradiction that $\cal F^+$
was dicritical.  Then, if we let $\cal G$ be the induced foliation
on $Z$ we must have that $\cal G$ is dicritical.  Let $E$ be a divisor sitting
over $Z$ which is not foliation invariant.  Without loss of generality we may assume that $E$
maps to a point in $Z$.

Let $W$ be the centre of $E$ on $X$.  $W$ cannot be a divisor since $X \rightarrow Z$ is small,
and since $\cal F$ is non-dicritical $W$ cannot be a point or a curve tangent to $\cal F$.
Thus $W$ is a curve transverse to $\cal F$, but which is nevertheless contracted by $\pi$.
This is our desired contradiction.
\end{proof}

Putting all this together:

\begin{theorem}
Let $\cal F$ be a co-rank 1 toric foliation with canonical and non-dicritical foliation singularities
on a toric variety $X$.
Then the MMP for $\cal F$ exists, and
ends either with a foliation where $K_{\cal F}$ is nef,
or with a fibration $\pi: X \rightarrow Z$ and $\cal F$ is pulled
back from a foliation on $Z$. Furthermore, only curves tangent to the foliation
are contracted in this MMP.
\end{theorem}
\begin{proof}
If $K_{\cal F}$ is not nef, there is an extremal ray on which $K_{\cal F}$
is negative.  We can contract this ray resulting in a either:

\begin{enumerate}
\item a fibration, in which case we stop;

\item a divisorial contraction, in which case we repeat with
the new variety;

\item a flipping contraction, in which case we perform the flip.
\end{enumerate}

We can have only finitely many steps of type 1 or 2.  Notice
that the foliation discrepancy of some place will always increase under a flip and that there
are only finitely many toric models which can be reached by a sequence of flips.
Thus there can be no infinite sequence of flips.
\end{proof}

\bibliography{math.bib}
\bibliographystyle{alpha}

\end{document}